\documentclass[11pt,letterpaper]{amsart}

\usepackage{graphicx}
\usepackage[centertags]{amsmath}
\usepackage{amsfonts}
\usepackage{amssymb}
\usepackage{amsthm}
\usepackage{newlfont}
\usepackage{hyperref}
\usepackage{tikz}
\usetikzlibrary{patterns}

\newtheorem{thm}{Theorem}[section]
\newtheorem*{wq}{Weiss' Question \cite{W84}}
\newtheorem{cor}[thm]{Corollary}
\newtheorem{lem}[thm]{Lemma}
\newtheorem{prop}[thm]{Proposition}

\theoremstyle{definition}
\newtheorem{defn}[thm]{Definition}

\theoremstyle{remark}

\numberwithin{equation}{section}

\newcommand{\R}{\mathbb{R}}
\newcommand{\N}{\mathbb{N}}
\newcommand{\Z}{\mathbb{Z}}
\newcommand{\Q}{\mathbb{Q}}
\newcommand{\cU}{\mathcal{U}}
\newcommand{\cV}{\mathcal{V}}
\newcommand{\cW}{\mathcal{W}}

\newcommand{\symd}{\triangle}

\newcommand{\res}{\restriction}

\newcommand{\acts}{\curvearrowright}

\newcommand{\from}{\colon}

\newcommand{\Basdim}{\mathrm{asdim}_{\textbf{B}}}
\newcommand{\asdim}{\mathrm{asdim}}

\newcommand{\asi}{\mathrm{asi}}
\newcommand{\walkeq}[3][]{\mathfrak{F}^{#1}_{#3}({#2})}

\newcommand{\colorA}{cyan}
\newcommand{\colorB}{green}
\newcommand{\colorC}{pink}
\newcommand{\colorAtext}{blue}
\newcommand{\colorBtext}{\colorB}
\newcommand{\colorCtext}{\colorC}

\begin{document}

\title[Borel asymptotic dimension and hyperfinite relations]{Borel asymptotic dimension and hyperfinite equivalence relations}

\author[Conley]{Clinton T. Conley}
\author[Jackson]{Steve C. Jackson}
\author[Marks]{Andrew S. Marks}
\author[Seward]{Brandon M. Seward}
\author[Tucker-Drob]{Robin D. Tucker-Drob}

\thanks{This paper was produced as the result of a SQuaRE workshop supported by the American Institute of Mathematics. The first author was supported by NSF grant DMS-1855579, the second author was supported by NSF grant DMS-1800323, the third author was supported by NSF grant DMS-2054182, the fourth author was supported by NSF grant DMS-1955090, and the fifth author was supported by NSF grant DMS-2216533.}

\begin{abstract}
A long standing open problem in the theory of hyperfinite equivalence relations asks if the orbit equivalence relation generated by a Borel action of a countable amenable group is hyperfinite. In this paper we prove that this question always has a positive answer when the acting group is polycyclic, and we obtain a positive answer for all free actions of a large class of groups including the lamplighter group $\Z_2 \wr \Z$ and all virtually solvable groups having finite Pr{\"u}fer rank. This marks the first time that a group of exponential volume-growth has been verified to have this property. In obtaining this result we introduce a new tool for studying Borel equivalence relations by extending Gromov's notion of asymptotic dimension to the Borel setting. We show that countable Borel equivalence relations of finite Borel asymptotic dimension are hyperfinite, and more generally we prove under a mild compatibility assumption that increasing unions of such equivalence relations are hyperfinite. As part of our main theorem, we prove for a large class of solvable groups that all of their free Borel actions have finite Borel asymptotic dimension (and finite dynamic asymptotic dimension in the case of a continuous action on a zero-dimensional space). We also provide applications to Borel chromatic numbers, Borel and continuous F{\o}lner tilings, topological dynamics, and $C^*$-algebras.
\end{abstract}

\maketitle

\section{Introduction}

The asymptotic dimension $\asdim(X, \rho)$ of a metric space $(X, \rho)$ is a large-scale analog of Lebesgue covering dimension. It was first introduced by Gromov in 1993 as a quasi-isometry invariant of finitely generated groups \cite{G93} (for a survey of asymptotic dimension and its applications to group theory, see \cite{BD08}). In recent years, a related notion of \emph{dynamic asymptotic dimension} was introduced in the context of topological dynamical systems by Guentner, Willet, and Yu \cite{GWY17}. While a few of our results pertain to topological dynamics and dynamic asymptotic dimension, our main focus is to further adapt the notion of asymptotic dimension to the context of Borel equivalence relations, and to investigate its applications.

Let $X$ be a standard Borel space, let $\rho : X \times X \rightarrow [0, +\infty]$ be a Borel extended metric on $X$, where by \emph{extended} we mean that the value $+\infty$ is allowed, and define the Borel equivalence relation $E_\rho = \{(x, y) \in X \times X : \rho(x, y) < \infty\}$. When every class of $E_\rho$ is countable, we define the \emph{Borel asymptotic dimension} of $(X, \rho)$, denoted $\Basdim(X, \rho)$, to be $d$ if $d \in \N \cup \{+\infty\}$ is least with the property that for every radius $r > 0$ there exists a Borel equivalence relation $F \subseteq E_\rho$ such that the $F$-classes have uniformly bounded finite diameter and every ball of radius $r$ meets at most $d+1$ many classes of $F$. The original notion of asymptotic dimension, which we will refer to as \emph{standard asymptotic dimension} for clarity, is defined in the same manner but without the Borel measurability restrictions. In particular, we always have $\Basdim(X, \rho) \geq \asdim(X, \rho)$.

It is perhaps a bit surprising that in the Borel context the most pertinent information is whether or not the Borel asymptotic dimension is finite, as the following theorem reveals. For the statement below, recall that a metric $\rho$ is \emph{proper} if every ball of finite radius has finite cardinality.

\begin{thm} \label{intro:equal}
Let $X$ be a standard Borel space and let $\rho$ be a Borel proper extended metric on $X$. If the Borel asymptotic dimension of $(X, \rho)$ is finite, then it is equal to the standard asymptotic dimension of $(X, \rho)$.
\end{thm}

We also find that Borel and standard asymptotic dimension agree modulo a meager set.

\begin{thm} \label{intro:bp}
Let $X$ be a Polish space and let $\rho$ be a Borel proper extended metric on $X$. Then there is an $E_\rho$-invariant comeager set $X' \subseteq X$ such that $\Basdim(X', \rho) = \asdim(X', \rho)$.
\end{thm}

In this paper, our primary focus will be the Borel asymptotic dimensions associated with Borel group actions and, ultimately, the subsequent applications to the theory of hyperfinite equivalence relations. If $G$ is a countable group and $\tau : G \times G \rightarrow [0, +\infty)$ is a proper right-invariant metric on $G$, then to every Borel action $G \acts X$ of $G$ on a standard Borel space $X$ one can associate a Borel proper extended metric $\rho_{\tau} : X \times X \rightarrow [0, +\infty]$ by declaring $\rho_{\tau}(x, y)$ to be the minimum of $\{\tau(1_G, g) : g \in G, \ g \cdot x = y\}$ when this set is non-empty, and $\infty$ otherwise.

For a countable group $G$, the standard asymptotic dimension of $(G, \tau)$ does not depend upon the choice of the proper right-invariant metric $\tau$ \cite[Prop. 62]{BD08} and is referred to as the \emph{asymptotic dimension of $G$}. Similarly, for a Borel action $G \acts X$ the Borel asymptotic dimension of $(X, \rho_\tau)$ does not depend on the choice of $\tau$ (see Lemma \ref{lem:asdim_action}). To simplify terminology, we will therefore speak of the Borel asymptotic dimension of the action $G \acts X$ and write $\Basdim(G \acts X)$.

Our main theorem is below. Recall that a \emph{normal series} for a group $G$ is a sequence $G = G_0 \rhd G_1 \rhd \ldots \rhd G_n = \{1_G\}$ of normal subgroups of $G$. We refer to $G_k / G_{k+1}$ as a quotient of consecutive terms, and we call $G_0 / G_1$ the top quotient. The standard asymptotic dimension of $G$ is always bounded by the sum of the standard asymptotic dimensions of the consecutive quotients in a normal series (see Thm. 68 in \cite{BD08} and the remark following it). 

\begin{thm} \label{intro:polyasdim}
Let $G$ be a countable group admitting a normal series where each quotient of consecutive terms is a finite group, an increasing union of characteristic finite subgroups, or a torsion-free abelian group with finite $\Q$-rank, except the top quotient can be any group of uniform local polynomial volume-growth or the lamplighter group $\Z_2 \wr \Z$. If $X$ is a standard Borel space and $G \acts X$ is a free Borel action, then $\Basdim(G \acts X) = \asdim(G) < \infty$. Additionally, if $X$ is a locally compact $0$-dimensional second countable Hausdorff space and $G$ acts freely and continuously on $X$, then the dynamic asymptotic dimension of $G \acts X$ is at most $\asdim(G) < \infty$.
\end{thm}

The above theorem clearly applies to the lamplighter group and all groups of uniform local polynomial volume-growth. It also applies, by Lemma \ref{lem:prufer} below, to all groups that are virtually solvable and have finite Pr{\"u}fer rank (a group has finite Pr{\"u}fer rank if there is $k$ so that every finitely generated subgroup can be generated by $k$ elements or fewer). The class of virtually sovlable groups having finite Pr{\"u}fer rank contains all polycyclic groups (i.e. groups admitting a normal series where every quotient of consecutive terms is a finitely-generated abelian group), all solvable groups that are linear over $\Q$, the Baumslag--Solitar group $BS(1,2)$, as well as some more exotic groups which are not virtually torsion-free nor linear over any field \cite{LenRob}.

Theorem \ref{intro:polyasdim} and the theorem below provide many new examples, both in the Borel and the topological settings, of groups having the property that all of their free actions admit F{\o}lner tilings. In view of the assumption of the theorem below, we remark that both Borel and dynamic asymptotic dimension are monotone decreasing upon restricting an action to the action of a subgroup (see for instance Corollary \ref{cor:subgroups}). In particular, the assumption of the next theorem will be met provided $G$ satisfies (or all of its finitely generated subgroups satisfy) the assumption of Theorem \ref{intro:polyasdim}.

\begin{thm} \label{intro:folner}
Let $G$ be a countable amenable group, let $X$ be a standard Borel space, and let $G \acts X$ be a free Borel action. Assume that $\Basdim(H \acts X) < \infty$ for every finitely generated subgroup $H \leq G$. Then for every finite $K \subseteq G$ and $\delta > 0$ there exist $(K, \delta)$-invariant finite sets $F_1, \ldots, F_n \subseteq G$ and Borel sets $C_1, \ldots, C_n \subseteq X$ such that the map $\theta : \bigsqcup_{i=1}^n F_i \times C_i \rightarrow X$ given by $\theta(f, c) = f \cdot c$ is a bijection. Similarly, if $X$ is a compact $0$-dimensional second countable Hausdorff space, $G$ acts continuously and freely on $X$, and the dynamic asymptotic dimension of $H \acts X$ is finite for every finitely generated subgroup $H \leq G$, then the same conclusion holds with the sets $C_1, \ldots, C_n$ being additionally clopen. 
\end{thm}

The clopen F{\o}lner tilings provided by Theorems \ref{intro:polyasdim} and \ref{intro:folner} immediately yield interesting consequences to topological dynamics and $C^*$-algebras, most notably providing new examples of classifiable crossed products. 

\begin{cor}
Let $G$ be a countably infinite group such that the assumptions of Theorem \ref{intro:polyasdim} are satisfied by $G$ or are satisfied by every finitely generated subgroup of $G$. Then every free continuous action of $G$ on a compact metrizable space having finite covering dimension is almost finite. As a consequence, the crossed products arising from minimal such actions are classified by the Elliot invariant (ordered $K$-theory paired with tracial states) and are simple ASH algebras of topological dimension at most $2$.
\end{cor}

Finite Borel asymptotic dimension also provides a bound on the Borel chromatic number $\chi_{\textbf{B}}(\Gamma)$ of a Borel graph $\Gamma$ in terms of its standard chromatic number $\chi(\Gamma)$. The theorem below follows a theme of similar results. Specifically, Conley and Miller established the same bound for Baire-measurable chromatic numbers of locally finite Borel graphs, and when the graph is furthermore hyperfinite they obtained this bound for $\mu$-measurable chromatic numbers \cite{CM16}. More recently, Gao--Jackson--Krohne--Seward obtained this bound for the Borel chromatic numbers of Schreier graphs induced by Borel actions of finitely-generated nilpotent groups \cite{GJKSc}.

\begin{thm} \label{intro:chi}
Let $X$ be a standard Borel space, let $\Gamma$ be a Borel graph on $X$, and let $\rho$ be the graph metric on $\Gamma$. Assume that $\Gamma$ is locally finite and that $\Basdim(X, \rho) < \infty$. Then $\chi_\textbf{B}(\Gamma) \leq 2 \cdot \chi(\Gamma) - 1$.
\end{thm}

In fact we obtain a result more general than the above by incorporating a new quantity we define called the \emph{asymptotic separation index}.

From the above theorem and previous work of the authors \cite{CJMST-D}, it follows that there are hyperfinite bounded degree acyclic Borel graphs having infinite Borel asymptotic dimension. This is in contrast to the fact that acyclic graphs always have standard asymptotic dimension at most $1$. On the other hand, in Lemma \ref{lem:bound_to_one} we prove that the Borel asymptotic dimension of a Borel graph is at most $1$ when the graph is induced by a bounded-to-one Borel function.

Lastly, we discuss applications to the theory of hyperfinite equivalence relations, which was our primary motivation in this work. Recall that an equivalence relation $E$ on a standard Borel space $X$ is \emph{Borel} if $E$ is a Borel subset of $X \times X$, and \emph{finite} (or \emph{countable}) if every $E$-class is finite (respectively countable). $E$ is \emph{hyperfinite} if it is the union of an increasing sequence of finite Borel equivalence relations. From each Borel action $G \acts X$ of a countable group $G$ we obtain the \emph{orbit equivalence relation} $E_G^X = \{(x, y) \in X \times X : \exists g \in G \ g \cdot x = y\}$, which is a countable Borel equivalence relation.

The study of Borel actions of countable groups is of fundamental significance to the theory of countable Borel equivalence relations, as every countable Borel equivalence relation can be represented as the orbit equivalence relation of a Borel action of a countable group \cite{FM77}. Additionally, under the hierarchy of Borel reducibility, the hyperfinite equivalence relations are the simplest class of countable Borel equivalence relations whose study is non-trivial (as made precise by the Glimm-Effros dichotomy and the equivalence of hyperfiniteness and Borel reducibility to $E_0$; see \cite[Thm. 7.1]{DJK94}). Thus, a fundamental problem is to determine for which actions $G \acts X$ is the orbit equivalence relation $E_G^X$ hyperfinite. A long-standing and widely-known formulation of this problem is the following.
\begin{wq}
If $G$ is a countable amenable group, $X$ a standard Borel space, and $G \acts X$ a Borel action, must the orbit equivalence relation $E_G^X$ be hyperfinite?
\end{wq}
Every countable non-amenable group $G$ admits a Borel action $G \acts X$ for which $E_G^X$ is not hyperfinite \cite{JKL02}. Thus, if Weiss' question has a positive answer, then it provides a characterization for the class of amenable groups. 

In a striking result in 1980, Ornstein and Weiss proved that for every Borel action $G \acts X$ of a countable amenable group $G$ the orbit equivalence relation $E_G^X$ is $\mu$-almost-everywhere hyperfinite for every Borel probability measure $\mu$ on $X$\footnote{The published version of their result is stated for non-atomic quasi-invariant probability Borel measures, but can be generalized to any Borel probability measure $\mu$ as follows. The collection of points containing an atom in their orbit is countable, and the restriction of $E_G^X$ to this set is easily seen to be smooth and thus hyperfinite. We can therefore discard that set and assume $\mu$ is non-atomic. Pick any enumerate $g_i$ of $G$ and set $\nu = \sum 2^{-i} (g_i)_* \mu$. Then $\nu$ is non-atomic and quasi-invariant, so there is a $G$-invariant $\nu$-conull set $Y$ with the restriction of $E_G^X$ to $Y$ hyperfinite. Now simply notice that $Y$ is $\mu$-conull as well.}, meaning for each $\mu$ there is a $G$-invariant $\mu$-conull set $Y \subseteq X$ such that the restriction of $E_G^X$ to $Y$ is hyperfinite \cite{OW80}. It was in light of this result that Weiss posed his question in \cite{W84} and proceeded, in that same paper, to provide a positive answer to his question for $\Z$ (a second proof for $\Z$ was published by Slaman and Steel in 1988 \cite{SS88}).

Since the conception of Weiss' question, significant attention has gone to the project of expanding the class of groups known to have a positive answer. In unpublished work shortly after posing his question, Weiss proved the answer is positive for the groups $\Z^n$. Years later, in 2002, Jackson, Kechris, and Louveau generalized this result by proving that $E_G^X$ is hyperfinite whenever $G$ is a finitely-generated group of polynomial volume-growth \cite[Thm. 1.16]{JKL02} (equivalently, whenever $G$ has a finite-index nilpotent subgroup \cite{G81,Wo68}). In an extremely technical and long proof that introduced important new methods, Gao and Jackson obtained a positive answer for all countable abelian groups in 2007 \cite{GJ15}. Through a further expansion of those methods, a positive answer was obtained for all locally nilpotent groups by Schneider and Seward in 2015 \cite{SS} (a group is \emph{locally nilpotent} if every finitely-generated subgroup is nilpotent).

These last two results (of Gao--Jackson and Schneider--Seward) can be equivalently described as partially removing the finite-generation assumption from the theorem of Jackson--Kechris--Louveau. Indeed, the classes of groups (abelian and locally nilpotent) considered in these two works have the property that all their finitely-generated subgroups have polynomial volume-growth. The assumption of polynomial volume-growth is a universal trait of prior solutions to Weiss' question and is fundamental to the structure of all prior proofs. For this reason, reliance upon polynomial volume-growth has widely been viewed as a critical obstruction to further progress.

Through Theorem \ref{intro:polyasdim} we obtain the first solution to Weiss' question that includes groups of exponential volume-growth (any solvable group not containing a finite-index nilpotent subgroup has exponential volume-growth \cite{Wo68}). We obtain this result by relating Borel asymptotic dimension to hyperfiniteness.

\begin{thm} \label{intro:asdimhyp}
Let $X$ be a standard Borel space and let $\rho$ be a Borel proper extended metric on $X$. If the Borel asymptotic dimension of $(X, \rho)$ is finite, then the finite-distance equivalence relation $E_\rho = \{(x, y) \in X \times X : \rho(x, y) < \infty\}$ is hyperfinite.
\end{thm}

\begin{cor} \label{intro:hyp1}
If $G$ is a countable group satisfying the assumption of Theorem \ref{intro:polyasdim}, $X$ is a standard Borel space, and $G \acts X$ is a free Borel action, then the orbit equivalence relation $E_G^X$ is hyperfinite. 	
\end{cor}

In the case of polycyclic groups, work of Schneider and Seward (\cite[Cor. 8.2]{SS}) allows us to remove the assumption that the actions are free. We thus obtain a positive answer to Weiss' question for all polycyclic groups.

\begin{cor} \label{intro:polyhyp}
If $G$ is a polycyclic group, $X$ is a standard Borel space, and $G \acts X$ is any Borel action, then the orbit equivalence relation $E_G^X$ is hyperfinite.
\end{cor}

As previously mentioned, the works of Gao--Jackson \cite{GJ15} and Schneider--Seward \cite{SS} saw a tremendous increase in the length and technicality of their arguments only to partially remove the finite-generation assumption in the work of Jackson--Kechris--Louveau \cite{JKL02}. A compelling reason for why this occurred is that removing finite-generation is equivalent to confronting a special case of another long-standing fundamental open problem on hyperfinite equivalence relations, one that is generally considered harder than Weiss' question. Specifically, the \emph{Union Problem} asks: if $X$ is a standard Borel space and $(E_n)$ is an increasing sequence of hyperfinite equivalence relations on $X$, must the equivalence relation $\bigcup_n E_n$ be hyperfinite? If we express a group $G$ as an increasing union of finitely generated subgroups $G_n$, then the overlap between this problem and Weiss' question can be seen by observing that for every action $G \acts X$ we have $E_G^X = \bigcup_n E_{G_n}^X$.

Using the notion of Borel asymptotic dimension, we solve a special case of the Union Problem.

\begin{thm} \label{intro:union}
Let $X$ be a standard Borel space and let $(\rho_n)_{n \in \N}$ be a sequence of Borel proper extended metrics on $X$. Assume that for every $n \in \N$ and $r > 0$ the $\rho_n$-balls of radius $r$ have uniformly bounded $\rho_{n+1}$-diameter, and set $E = \bigcup_n E_{\rho_n}$. If $(X, \rho_n)$ has finite Borel asymptotic dimension for every $n$ then $E$ is hyperfinite.
\end{thm}

The above theorem allows us to immediately expand on Corollary \ref{intro:hyp1}.

\begin{cor}
If $G$ is the increasing union of groups satisfying the assumption of Theorem \ref{intro:polyasdim}, $X$ is a standard Borel space, and $G \acts X$ is a free Borel action, then the orbit equivalence relation $E_G^X$ is hyperfinite.
\end{cor}

Theorem \ref{intro:union} also allows us to, for the first time, fully remove the finite-generation assumption from the work of Jackson--Kechris--Louveau \cite{JKL02}. Indeed, it is simple to prove (as we do in Corollary \ref{cor:polygrow_asdim}) that Borel actions $G \acts X$ have finite Borel asymptotic dimension whenever $G$ is a finitely-generated group of polynomial volume-growth.

\begin{cor}
Let $G$ be a countable group with the property that every finitely generated subgroup of $G$ has polynomial volume-growth (equivalently, every finitely generated subgroup of $G$ contains a finite-index nilpotent subgroup). If $X$ is a standard Borel space and $G \acts X$ is any Borel action then the orbit equivalence relation $E_G^X$ is hyperfinite.
\end{cor}

For instance, a new group to which the above corollary applies but prior results do not is the group $(\bigoplus_{n \in \N} \Z) \rtimes S_\infty$, where $S_\infty$ is the group of finitely-supported permutations of $\N$.

In addition to the formal results we obtain, we believe that the methods we employ are equally valuable. The introduction of Borel asymptotic dimension into the study of hyperfinite equivalence relations is exciting both for the new results it leads to, as discussed above, and for how it reshapes our understanding of prior results. Specifically, the papers by Gao--Jackson \cite{GJ15} and Schneider--Seward \cite{SS}, which were primarily devoted to partially removing the finite-generation assumption in the work of Jackson--Kechris--Louveau \cite{JKL02}, were quite long (74 and 45 pages, respectively) and highly technical. Theorem \ref{intro:union} achieves the chief goal of those papers, but its proof is less technical and tremendously shorter (only a few pages, including all supporting lemmas). Together with Corollary \ref{cor:polygrow_asdim} (which also has a very short proof), these self-contained arguments encompass all prior work on Weiss' question and a bit more. Our proofs certainly differ significantly in terms of formal concepts and details, however we find it both surprising and satisfying that, in our opinion, our arguments capture the core intuitive components of those prior works. We hope therefore that this work can lead to a renewed understanding of prior discoveries and can reinvigorate the pursuit of Weiss' question.

\subsection*{Outline}
Section \ref{sec:prelim} contains a review of preliminary material. In Section \ref{sec:asdim} we introduce Borel asymptotic dimension and asymptotic separation index. Many of the remaining sections are mostly independent of one another. In Section \ref{sec:general} we prove that Borel asymptotic dimension is equal to standard asymptotic dimension when the former is finite, and we prove that in many situations the asymptotic separation index is at most $1$. In the first half of Section \ref{sec:polygrow} we retrace part of the work of Jackson--Kechris--Louveau, but in the language of Borel asymptotic dimension, and prove that Borel actions of groups of polynomial volume-growth have finite Borel asymptotic dimension. In the second half of Section \ref{sec:polygrow} we study packing phenomena for certain group extensions and, relying upon that, in Section \ref{sec:series} we prove our main theorem that all free actions of groups with suitable normal series have finite Borel asymptotic dimension. In Section \ref{sec:hyperfinite} we prove that finite Borel asymptotic dimension implies hyperfiniteness and, relying a bit on Section \ref{sec:general}, we use Borel asymptotic dimension to solve an instance of the Union Problem. Finally, we study applications to Borel chromatic numbers in Section \ref{sec:chromatic}, F{\o}lner tilings in Section \ref{sec:folner}, and topological dynamics and $C^*$-algebras in Section \ref{sec:calg}. Much of this last section discusses alterations to proofs from earlier in the paper, arguing how clopen sets can be used in place of Borel sets.

\section{Preliminaries} \label{sec:prelim}

As is customary in logic, throughout the paper we identify each positive integer $n$ with the set $\{0, 1, \ldots, n-1\}$.

By an \emph{extended} metric we mean a metric that is allowed to take value $+\infty$. Given an extended metric $\rho$ on $X$ we write $E_\rho = \{(x, y) \in X \times X : \rho(x, y) < \infty\}$ for the finite-distance equivalence relation. We let $B(x; r)$ denote the closed $r$-ball centered at $x$, and we call $\rho$ \emph{proper} if every ball of finite radius has finite cardinality.

\subsection{Groups and group actions}
A \emph{normal series} for a group $G$ is a sequence $G = G_0 \rhd G_1 \rhd \ldots \rhd G_n = \{1_G\}$ of normal subgroups of $G$. We refer to $G_k / G_{k+1}$ as a \emph{quotient of consecutive terms}, and we call $G_0 / G_1$ the \emph{top quotient}.

If $G$ is an abelian group, then its \emph{torsion subgroup} is the subgroup consisting of all elements of finite order. $G$ is called \emph{torsion} if it is equal to its torsion subgroup, and $G$ is called \emph{torsion-free} if its torsion subgroup is trivial. The \emph{$\Q$-rank} of $G$ (when $G$ is written additively) is the largest cardinality of a $\Z$-linearly independent subset, or equivalently the dimension of the $\Q$-vector space $\Q \otimes G$.

If $P$ is a property of groups, then we will say that $G$ is \emph{locally} $P$ if every finitely generated subgroup of $G$ has property $P$, and we will say $G$ is \emph{virtually} $P$ if $G$ contains a finite-index subgroup with property $P$.

A group $G$ has \emph{finite Pr{\"u}fer rank} if there is $k \in \N$ so that every finitely generated subgroup of $G$ can be generated by $k$ elements or fewer. When this is the case, the \emph{Pr{\"u}fer rank} of $G$ is the least $k$ for which this property holds.

\begin{lem} \label{lem:prufer}
Let $G$ be a virtually solvable group having finite Pr{\"u}fer rank. Then $G$ admits a normal series where each quotient of consecutive terms is either a finite group, an increasing union of finite characteristic subgroups, or a torsion-free abelian group with finite $\Q$-rank.
\end{lem}

\begin{proof}
From the definition it is immediate that the property of having finite Pr{\"u}fer rank passes from a group to all of its subgroups. Additionally, it passes to all quotients. Specifically, suppose that $\Gamma$ has Pr{\"u}fer rank $k$ and $\phi : \Gamma \rightarrow \Lambda$ is an epimorphism. Given any finite collection $\lambda_1, \ldots, \lambda_\ell \in \Lambda$ we can pick $\gamma_1, \ldots, \gamma_\ell \in \Gamma$ satisfying $\phi(\gamma_i) = \lambda_i$ for each $i$. Since $\Gamma$ has Pr{\"u}fer rank $k$ there are $s_1, \ldots, s_k \in \Gamma$ with $\langle s_1, \ldots, s_k \rangle = \langle \gamma_1, \ldots, \gamma_\ell \rangle$. It follows that $\langle \phi(s_1), \ldots, \phi(s_k) \rangle = \langle \lambda_1, \ldots, \lambda_\ell \rangle$. Thus every finitely generated subgroup of $\Lambda$ is at most $k$-generated and $\Lambda$ has Pr{\"u}fer rank at most $k$.

Next, we note that if $\Gamma$ is a torsion abelian group having finite Pr{\"u}fer rank, then $\Gamma$ is the increasing union of characteristic finite subgroups. Indeed, $\Gamma$ is the increasing union of the characteristic subgroups $\Gamma_n$, where $\Gamma_n$ is the subgroup generated by all elements having order $n$ or smaller. Each element of $\Gamma_n$ has order at most $n!$ by commutativity, and if $k$ is the Pr{\"u}fer rank of $\Gamma$ then every finitely generated subgroup of $\Gamma_n$ is $k$-generated and thus has cardinality at most $(n!)^k$. Therefore $|\Gamma_n| \leq (n!)^k$ and $\Gamma_n$ is finite as desired.

We now prove the lemma. Let $H$ be a finite-index solvable subgroup of $G$. The left-translation action of $G$ on the space of cosets $G / H$ produces a homomorphism $\phi : G \rightarrow \mathrm{Sym}(G/H)$. The kernel of this map, $\ker \phi = \bigcap_{g \in G} g H g^{-1}$, is a subgroup of $H$ and thus solvable, and it has finite index in $G$ since $\mathrm{Sym}(G/H)$ is finite.

Set $G_0 = G$ and $G_1 = \ker \phi$. Let $i \geq 1$ and suppose that $G_{2i-1}$ has been defined. Set $G_{2i+1} = [G_{2i-1}, G_{2i-1}]$ and let $G_{2i}$ be the set of elements $g \in G_{2i-1}$ such that $g^m \in G_{2i+1}$ for some $m > 0$. It is easily seen by induction that $G_j$ is characteristic in $G_1$ for all $j \geq 1$,  and thus $(G_j)$ is a normal series for $G$. Additionally, since $(G_{2i-1})_{i \geq 1}$ is the derived series of $G_1$ and $G_1$ is solvable, there is $n \geq 1$ with $G_n = \{1_G\}$.

By construction $G_0 / G_1$ is a finite group. For $i \geq 1$, $G_{2i} / G_{2i+1}$ is the torsion subgroup of the abelian group $G_{2i-1} / G_{2i+1}$, so $G_{2i-1} / G_{2i}$ is torsion-free and abelian. Additionally, since $G_{2i-1} / G_{2i}$ has finite Pr{\"u}fer rank it has finite $\Q$-rank. Lastly, $G_{2i} / G_{2i+1}$ is a torsion abelian group having finite Pr{\"u}fer rank and thus is the increasing union of finite characteristic subgroups.
\end{proof}

A finitely generated group $G$ has \emph{polynomial volume-growth} if there is $d \in \N$ (the least of which is called the \emph{degree}) so that for any (equivalently every) finite symmetric generating set $B$ containing $1_G$ there is $c > 0$ with $|B^r| \leq c r^d$ for all $r \geq 1$. By Gromov's theorem \cite{G81,Wo68}, this is equivalent to $G$ being finitely generated and virtually nilpotent.

We say that $G$ has \emph{local polynomial volume-growth} if every finitely generated subgroup of $G$ has polynomial volume-growth, and has \emph{uniform local polynomial volume-growth} if there is $d \in \N$ (the least of which is referred to as the \emph{degree}) so that every finitely generated subgroup of $G$ has polynomial volume-growth of degree at most $d$. For instance, locally finite groups have uniform local polynomial volume-growth of degree $0$ and countable abelian groups of $\Q$-rank $d$ have uniform local polynomial volume-growth of degree $d$.

Every finitely generated group $G$ can naturally be viewed as a metric space. Specifically, to any finite symmetric generating set $A$ containing $1_G$ one can associate the \emph{$A$-word length metric} $\tau_A$ defined by letting $\tau_A(g, h)$ be the least $n \in \N$ with $g h^{-1} \in A^n$. Then $\tau_A$ is a proper right-invariant metric. More generally, every countable group $G$ admits some proper right-invariant metric. For instance, if $G$ is finite we can let $\tau$ be the discrete metric, and otherwise we can fix a bijection $w : \Z_+ \rightarrow G$ and define
$$\tau(g, h) = \min \{n_1 + n_2 + \ldots + n_k : k \geq 1, \ w(n_1) w(n_2) \cdots w(n_k) = g h^{-1}\}.$$

An action $G \acts X$ is \emph{free} if $g \cdot x \neq x$ for every $x \in X$ and non-identity $g \in G$. If $G$ acts on $X$ and $\tau$ is a proper right-invariant metric on $G$, then we can define a proper extended metric $\rho_\tau$ on $X$ by letting $\rho_\tau(x, y)$ be the minimum of $\tau(1_G, g)$ over all $g \in G$ satisfying $g \cdot x = y$, and be $+\infty$ if no such $g$ exists. Since $\tau$ is right-invariant, for any fixed $x \in X$ the map $g \mapsto g \cdot x$ from $(G, \tau)$ to $(X, \rho_\tau)$ is $1$-Lipschitz and is an isometry when the action is free.

\subsection{Standard asymptotic dimension}

Let $(X, \rho)$ be an extended metric space. We say that an equivalence relation $E$ on $X$ is \emph{bounded} if every $E$-class has finite $\rho$-diameter, and is \emph{uniformly bounded} if there is a uniform finite bound to the $\rho$-diameters of $E$-classes. We also say $E$ is \emph{countable} if every $E$-class is countable, and \emph{finite} if every $E$-class is finite. Moreover, $E$ is \emph{uniformly finite} if there is a finite uniform bound to the cardinalities of the $E$-classes.

For $U \subseteq X$ and $r > 0$, we denote by $\walkeq{U}{r}$ the smallest equivalence relation on $U$ with the property that $(x,y) \in \walkeq{U}{r}$ whenever $x, y \in U$ satisfy $\rho(x, y) \leq r$. Similarly, for an action $G \acts X$, $U \subseteq X$, and finite $A \subseteq G$, we let $\walkeq{U}{A}$ be the smallest equivalence relation on $U$ satisfying $(x, y) \in \walkeq{U}{A}$ whenever $x, y \in U$ and $x \in A \cdot y$ or $y \in A \cdot x$.

There are several equivalent definitions for the standard asymptotic dimension of a metric space $(X, \rho)$ (see for instance \cite[Thm. 19]{BD08}), but we will only be interested in two. Consider the following statements.
\begin{enumerate}
\item[(1)] For every $r > 0$ there is a uniformly bounded equivalence relation $E$ on $X$ such that for every $x \in X$ the ball $B(x; r)$ meets at most $d+1$ classes of $E$.
\item[(2)] For every $r > 0$ there are sets $U_0, \ldots, U_d$ that cover $X$ and have the property that $\walkeq{U_i}{r}$ is uniformly bounded for every $i$.	
\end{enumerate}
Then (1) and (2) are equivalent for every $d \in \N$ (this is a standard fact, recorded for instance in \cite[Thm. 19]{BD08} and reproved in the Borel setting in the next section). The \emph{standard asymptotic dimension} of $(X, \rho)$, denoted $\asdim(X, \rho)$, is the least $d \in \N$ for which (1) and (2) are true, or $+\infty$ if no such $d$ exists. We will use this same definition for extended metric spaces as well.

A visual demonstration, based on (1), that $\R^2$ has standard asymptotic dimension $2$ is depicted in Figure \ref{fig:plane}. In that figure, the classes of $E$ are rectangles of side-length greater than $4r$, and the rectangles are staggered so that every $r$-ball meets at most $3$ rectangles, rather than $4$. A similar demonstration, based on (2), is depicted in Figure \ref{fig:plane2}. The objects depicted in these figures are connected by a construction that will be discussed in the next section.

\begin{figure}
\begin{tikzpicture}
\foreach \x in {0,1,2,3,4} {
		\draw (1.5*\x,0)--(1.5*\x,1);
		\draw (1.5*\x,2)--(1.5*\x,3);
		}
\foreach \x in {0,1,2,3} {
		\draw (1.5*\x+0.75,0)--(1.5*\x+0.75,-0.3);
		\draw (1.5*\x+0.75,1)--(1.5*\x+0.75,2);
		\draw (1.5*\x+0.75,3)--(1.5*\x+0.75,3.3);
		}
\foreach \y in {0,1,2,3} {
		\draw (-0.3,\y)--(6+0.3,\y);
		}
\draw (-2,1.85) circle [radius=0.2];
\node [below] at (-2,1.65) {Ball of radius};
\node [below] at (-2,1.25) {$r = 4r_0$};
\end{tikzpicture}
\caption{An equivalence relation $E$ satisfying condition (1) when $X = \R^2$ and $d = 2$.} \label{fig:plane}
\end{figure}
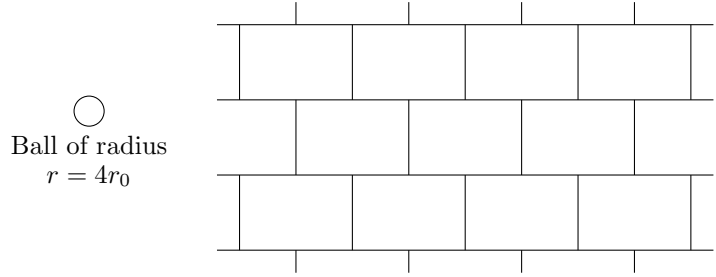

\begin{figure}
\begin{tikzpicture}
\draw [\colorC, fill=\colorC] (-0.375,-0.25) rectangle (6.375,3.25);
\foreach \y in {0,1,2,3} {
		\draw [\colorB, fill=\colorB] (-0.375,\y-0.08) rectangle (6.375,\y+0.08);
		\draw (-0.375,\y-0.08)--(6.375,\y-0.08);
		\draw (-0.375,\y+0.08)--(6.375,\y+0.08);
		}
\foreach \x in {0,1,2,3} {
		\draw [\colorB, fill=\colorB] (1.5*\x+0.75-0.08,0) rectangle (1.5*\x+0.75+0.08,-0.25);
		\draw [fill=\colorB] (1.5*\x+0.75-0.08,1) rectangle (1.5*\x+0.75+0.08,2);
		\draw [\colorB, fill=\colorB] (1.5*\x+0.75-0.08,3) rectangle (1.5*\x+0.75+0.08,3.25);
		\draw (1.5*\x+0.75-0.08,0)--(1.5*\x+0.75-0.08,-0.25); \draw (1.5*\x+0.75+0.08,0)--(1.5*\x+0.75+0.08,-0.25);
		\draw (1.5*\x+0.75-0.08,3)--(1.5*\x+0.75-0.08,3.25); \draw (1.5*\x+0.75+0.08,3)--(1.5*\x+0.75+0.08,3.25);
		}
\foreach \x in {0,1,2,3,4} {
		\draw [fill=\colorB] (1.5*\x-0.08,0) rectangle (1.5*\x+0.08,1);
		\draw [fill=\colorB] (1.5*\x-0.08,2) rectangle (1.5*\x+0.08,3);
		}
\foreach \x in {0,1,...,8} {
		\foreach \y in {0,1,2,3} {
				\draw [fill=\colorA, radius=0.15] (1.5*\x/2,\y) circle;
				}
		}
\draw (-2,1.85) circle [radius=0.05];
\node [below] at (-2,1.65) {Ball of radius};
\node [below] at (-2,1.25) {$r_0$};
\end{tikzpicture}
\caption{Sets $U_0$ (\colorCtext), $U_1$ (\colorBtext), $U_2$ (\colorAtext) satisfying condition (2) when $r = r_0$, $X = \R^2$, and $d = 2$.} \label{fig:plane2}
\end{figure}
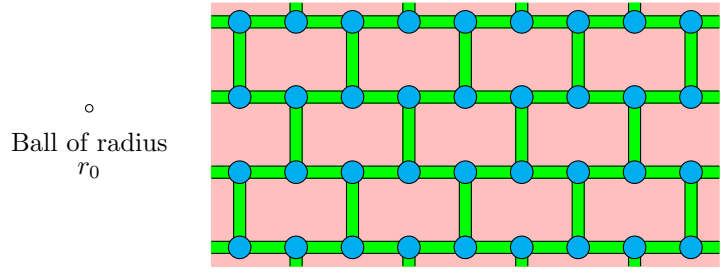

We define the standard asymptotic dimension of a group $G$, denoted $\asdim(G)$, and the standard asymptotic dimension of an action $G \acts X$, denoted $\asdim(G \acts X)$, to be $\asdim(G, \tau)$ and $\asdim(X, \rho_\tau)$, respectively, for any proper right-invariant metric $\tau$ on $G$. These two quantities do not depend on the choice of $\tau$.

\begin{lem} \label{lem:asdim_action}
Let $G \acts X$ be an action of a countable group $G$. Then $\asdim(G \acts X)$ is equal to the least $d \in \N$ with the property that for every finite set $A \subseteq G$ there are sets $U_0, \ldots, U_d$ covering $X$ such that $\walkeq{U_i}{A}$ is uniformly finite for every $i \in d+1$, and is equal to $\infty$ if no such $d$ exists.
\end{lem}

\begin{proof}
This follows from the basic fact that if $\tau$ is any proper right-invariant metric on $G$ then for any $r \geq 0$ the set $A = \{a \in G : \tau(a, 1_G) \leq r\}$ is finite and for any finite $A \subseteq G$ there is an $r \geq 0$ with $\tau(a, 1_G) \leq r$ for all $a \in A$. We leave it as an easy exercise for the reader to fill in the details.
\end{proof}

We note the following immediate consequence.

\begin{cor} \label{cor:subgroups}
Let $G \acts X$ be an action of a countable group $G$. If $H \leq G$ then the restricted action satisfies $\asdim(H \acts X) \leq \asdim(G \acts X)$. Moreover, $\asdim(G \acts X)$ is the supremum of $\asdim(H \acts X)$ as $H$ varies over all finitely generated subgroups of $G$.
\end{cor}

Similar statements hold for the standard asymptotic dimension of groups. To give a few examples, locally finite groups have standard asymptotic dimension $0$, $\Q^d$ has standard asymptotic dimension $d$, solvable groups have standard asymptotic dimension bounded by their Hirsch length (i.e. the sums of the $\Q$-ranks of the abelian quotients appearing in their derived series, which may be infinite), and non-abelian free groups have standard asymptotic dimension $1$. On the other hand, $\bigoplus_{n \in \N} \Z$, $\Z \wr \Z$, and the Grigorchuk group have infinite standard asymptotic dimension, as does any group having them as a subgroup. Further examples can be found in \cite{BD08}. It is also worth noting that every countable group has standard asymptotic dimension bounded by the sum of the standard asymptotic dimensions of the quotients of consecutive terms in any normal series (see Thm. 68 in \cite{BD08} and the remark following it).

\subsection{Borel equivalence relations and graphs}

A set $X$ equipped with a distinguished $\sigma$-algebra $\mathcal{B}(X)$ is called a \emph{standard Borel space} if there exists a separable and completely metrizable topology on $X$ such that $\mathcal{B}(X)$ coincides with the $\sigma$-algebra generated by the open sets. In this setting, we call the members of $\mathcal{B}(X)$ \emph{Borel} sets. When $X$ is a standard Borel space and $\rho$ is a Borel measurable extended metric on $X$, we call $(X, \rho)$ a \emph{Borel extended metric space}.

An equivalence relation $E$ on $X$ is \emph{Borel} if $E$ is a Borel subset of $X \times X$. A Borel equivalence relation $E$ is \emph{hyperfinite} if it can be expressed as the increasing union of a sequence of finite Borel equivalence relations $E_n$. For $x \in X$ we write $[x]_E$ for the $E$-class of $x$, and for $Y \subseteq X$ we write $[Y]_E = \{x \in X : \exists y \in Y \ (x, y) \in E\}$ for the \emph{$E$-saturation} of $Y$.

The equivalence relations we consider will always be countable. An important feature of countable Borel equivalence relations is that the Lusin-Novikov uniformization theorem provides Borel functions $f_n : X \rightarrow X$ with the property that $E = \{(x, f_n(x)) : x \in X, \ n \in \N\}$ \cite[Thm. 18.10]{K95}. This fact is quite useful for verifying that sets are Borel. For instance, if $E$ is a countable Borel equivalence relation and $Y \subseteq X$ is Borel, then $[Y]_E$ is Borel as well because if $(f_n)_{n \in \N}$ is as described above then $x \in [Y]_E$ if and only if there is $n \in \N$ with $f_n(x) \in Y$. The next lemma provides another example.

\begin{lem} \label{lem:walkborel}
Let $X$ be a standard Borel space and let $U \subseteq X$ be Borel.
\begin{enumerate}
\item If $\rho$ is a Borel extended metric on $X$ and $E_\rho$ is countable, then $\walkeq{U}{r}$ is Borel for every $r \geq 0$.
\item If $G \acts X$ is a Borel action then $\walkeq{U}{A}$ is Borel for every finite $A \subseteq G$.	
\end{enumerate}	
\end{lem}

\begin{proof}
Consider (a). Since $E_\rho$ is countable, we may let $f_n : X \rightarrow X$, $n \in \N$, be as given by the Lusin--Novikov uniformization theorem. Then $(x, y) \in \walkeq{U}{r}$ if and only if there is $k \in \N$ and $n_0, \ldots, n_k \in \N$ with $f_{n_0}(x) = x$, $f_{n_k}(x) = y$, $\rho(f_{n_i}(x), f_{n_{i+1}}(x)) \leq r$ for all $i < k$, and $f_{n_i}(x) \in U$ for all $i \leq k$. Thus $\walkeq{U}{r}$ is Borel. The proof of (b) is nearly identical, but one can use the elements of $G$ instead of the $f_n$'s.
\end{proof}

In our arguments the equivalence relations $\walkeq{U}{r}$ and $\walkeq{U}{A}$ will often be finite. A useful feature of finite Borel equivalence relations is that they admit \emph{Borel selectors}, meaning there is a Borel function $s \from X \rightarrow X$ satisfying $s(y) = s(x) \ E \ x$ for all $(x, y) \in E$ \cite[Thm. 12.16]{K95}. In particular, since $s$ is finite-to-one the image $s(X)$ is a Borel set that contains exactly one point from every $E$-class \cite[Ex. 18.14]{K95}. 

If $X$ is a standard Borel space, a \emph{Borel graph} on $X$ is a set of edges $\Gamma$ that is a Borel subset of $X \times X$. The \emph{graph metric} $\rho$ is defined by letting $\rho(x, y)$ be the shortest possible length of a path joining $x$ and $y$, or $\infty$ if no such path exists. This metric is proper when $\Gamma$ is \emph{locally finite}, meaning every $x \in X$ has finite degree in $\Gamma$. A set $Y \subseteq X$ is $\Gamma$-independent if no two points in $Y$ are joined by an edge, and a \emph{proper coloring} of $\Gamma$ is a function $c : X \rightarrow Z$ with the property that $c^{-1}(z)$ is $\Gamma$-independent for every $z \in Z$. The \emph{chromatic number} of $\Gamma$, denoted $\chi(\Gamma)$, is the least cardinality of a set $Z$ such that there exists proper coloring $c : X \rightarrow Z$. When $X$ is a standard Borel space and $\Gamma$ is a Borel graph, the \emph{Borel chromatic number} $\chi_{\textbf{B}}(\Gamma)$ is defined in the same way but with the requirement that $Z$ be a standard Borel space and that $c$ be Borel measurable. 

The following proposition gives a case when the Borel chromatic number is equal to the standard chromatic number:

\begin{prop}[{\cite[Proposition 1]{CM16}}]\label{lem:smooth_coloring}
Suppose that $X$ is a standard Borel space and $\Gamma$ is a locally countable Borel graph on $X$ whose connectedness relation $E_\Gamma$ on $X$ admits a Borel selector. Then $\chi(G) = \chi_B(G)$. 
\end{prop}

We will use this proposition in the case when the Borel graph $\Gamma$ has finite connected components and hence has a Borel selector by the above. In this special case, one way of proving this proposition is by taking a Borel linear ordering of $X$ and then taking the lexicographically least coloring using $\chi(G)$-colors in each connected component. It is not hard to verify this coloring is Borel.

\section{Borel Asymptotic Dimension} \label{sec:asdim}

In this section we introduce Borel asymptotic dimension. We first translate into the Borel setting the two conditions discussed in Section \ref{sec:prelim} that define standard asymptotic dimension, and we check that they are equivalent when $E_\rho$ is countable.

\begin{lem} \label{lem:equivalent_asym_dim}
Let $(X, \rho)$ be a Borel extended metric space with $E_\rho$ countable. Then for every $d \in \N$ the following statements are equivalent.
\begin{enumerate}
\item[(1)] For every $r > 0$ there is a uniformly bounded Borel equivalence relation $E$ such that for every $x \in X$ the ball $B(x; r)$ meets at most $d+1$ classes of $E$.
\item[(2)] For every $r > 0$ there are Borel sets $U_0, U_1, \ldots, U_d$ that cover $X$ and have the property that $\walkeq{U_i}{r}$ is uniformly bounded for every $i \in d+1$.
\end{enumerate}
Furthermore, if (1') and (2') are the statements where ``uniformly bounded'' is replaced by ``bounded'' in (1) and (2) respectively, then we similarly have that (1') and (2') are equivalent for every $d \in \N$.
\end{lem}

\begin{proof}
Since $E_\rho$ is countable, the Lusin--Novikov uniformization theorem provides Borel functions $f_n \from X \to X$, $n \in \N$, with $E_\rho = \{(x, f_n(x)) : x \in X, \ n \in \N\}$. We prove that (1) and (2) are equivalent; the equivalence of (1') and (2') follows the same argument but replacing ``uniformly bounded'' with ``bounded'' throughout.

(1) $\Rightarrow$ (2). Fix $r_0 > 0$. Apply (1) with $r = (d + 2) r_0$ to obtain a uniformly bounded Borel equivalence relation $E$ such that for every $x \in X$ the $(d + 2) r_0$-ball around $x$ meets at most $d + 1$ classes of $E$. For each $0 \leq i \leq d$ define $U_i$ to be the set of $x \in X$ having the property that both $B(x; (i+1) r_0)$ and $B(x; (i+2) r_0)$ meet $i+1$ many classes of $E$. Figure \ref{fig:plane2} depicts the sets $U_0, U_1, U_2$ that are obtained when this construction is applied to the relation $E$ in Figure \ref{fig:plane}. Notice that $B(x; r_0)$ must meet at least one $E$-class and, by our choice of $E$, $B(x; (d+2) r_0)$ must meet strictly less than $d+2$ many $E$-classes. Thus monotonicity properties imply that $\{U_i : 0 \leq i \leq d\}$ covers $X$. Additionally, each $U_i$ is Borel because for every $i, j$ the set $C_{i,j} = \{x \in X : B(x; (i+1) r_0) \text{ meets at least } j \text{ classes of } E\}$ is Borel, as $x \in C_{i,j}$ if and only if there are $n_1, \ldots, n_j \in \N$ with $(f_{n_s}(x), f_{n_t}(x)) \not\in E$ and $\rho(x, f_{n_s}(x)) \leq (i+1) r_0$ for all $s \neq t$.

Now suppose that $x, x' \in U_i$ and $\rho(x, x') \leq r_0$. We claim that for every class $D \in X / E$ we have that $B(x; (i+1) r_0)$ meets $D$ if and only if $B(x' ; (i+1) r_0)$ meets $D$. Indeed, if $B(x; (i+1) r_0)$ meets $D$ then by the triangle inequality $B(x' ; (i+2) r_0)$ meets $D$, and since $B(x' ; (i+1) r_0)$ and $B(x' ; (i+2) r_0)$ meet the same number of $E$-classes, we conclude that $B(x' ; (i+1) r_0)$ meets $D$. From this claim it follows that for every $x \in U_i$ the $\walkeq{U_i}{r_0}$-class of $x$ is contained in the ball of radius $(i+1) r_0$ around $[x]_E$. So $\walkeq{U_i}{r_0}$ is uniformly bounded since $E$ is uniformly bounded.

(2) $\Rightarrow$ (1). Fix $r_0 > 0$. Apply (2) with $r = 2 r_0$ to get Borel sets $U_0, U_1, \ldots, U_d$. By replacing each $U_i$ with $U_i \setminus \bigcup_{j < i} U_j$ we can assume that the $U_i$'s partition $X$. Then $E = \bigcup_i \walkeq{U_i}{2r_0}$ is a uniformly bounded Borel equivalence relation on $X$. For every $x$ and every $i$ we have that $B(x; r_0)$ can meet at most one class of $\walkeq{U_i}{2r_0}$, and thus $B(x; r_0)$ meets at most $d+1$ many classes of $E$.
\end{proof}

Since the non-Borel versions of (1) and (2) are both commonly used to define standard asymptotic dimension and these conditions remain equivalent in the Borel setting when $E_\rho$ is countable, we restrict our definition of Borel asymptotic dimension to this setting. We leave open how to define Borel asymptotic dimension in other situations.

\begin{defn} \label{defn:asdim}
Let $(X, \rho)$ be a Borel extended metric space such that $E_\rho$ is countable. The \emph{Borel asymptotic dimension} of $(X, \rho)$, denoted $\Basdim(X, \rho)$ is the least $d \in \N$ for which the statements (1) and (2) in Lemma \ref{lem:equivalent_asym_dim} are true, and is $\infty$ if no such $d$ exists. Similarly, the \emph{asymptotic separation index} of $(X, \rho)$, denoted $\asi(X, \rho)$ is the least $d \in \N$ for which the statements (1') and (2') are true, and is $\infty$ if no such $d$ exists.
\end{defn}

The notion of asymptotic separation index is a new concept that we introduce. We will see that this notion carries useful consequences, however it is unclear how rich this notion is. For one, the notion is useless when Borel constraints are removed as it is always at most $1$ (pick a single point from each $E_\rho$-class and use annuli of inner and outer radii $2rn$ and $2r(n+1)$, $n \in \N$, to determine the classes of $E$ as in (1')). Moreover, we will prove in the next section that the asymptotic separation index is often at most $1$, and we know of no example where it is both finite and greater than $1$.

For a Borel action $G \acts X$ of a countable group $G$, we can define the Borel asymptotic dimension $\Basdim(G \acts X)$ to be $\Basdim(X, \rho_\tau)$ for any choice of proper right-invariant metric $\tau$ on $G$. Once again this quantity does not depend on the choice of $\tau$ and the condition in Lemma \ref{lem:asdim_action} again characterizes the value $\Basdim(G \acts X)$ when the sets $U_i$ are required to be Borel. Similarly, Borel asymptotic dimensions of group actions have the same monotonicity properties described in Corollary \ref{cor:subgroups}. Similar statements apply to the asymptotic separation index of a Borel action $G \acts X$, but we will not need this.

\section{General properties} \label{sec:general}

In this section we establish some basic constraints on the values of Borel asymptotic dimension and asymptotic separation index. We first note that these values can be $0$ only in fairly trivial situations.

\begin{cor}
Let $(X, \rho)$ be a Borel extended metric space with $E_\rho$ countable. Then $\Basdim(X, \rho) = 0$ if and only if $\walkeq{X}{r}$ is uniformly bounded for every $r > 0$. Similarly, $\asi(X, \rho) = 0$ if and only if $\walkeq{X}{r}$ is bounded for every $r > 0$.
\end{cor}

\begin{proof}
This is immediate from conditions (2) and (2') of Lemma \ref{lem:equivalent_asym_dim}, as we must have $U_0 = X$.
\end{proof}

We next prove that Borel asymptotic dimension can only differ from its standard counterpart when it is infinite.

\begin{thm} \label{thm:equal}
Suppose that $(X,\rho)$ is a proper Borel extended metric space. If $\asi(X, \rho) < \infty$ (or, less generally, if $\Basdim(X, \rho) < \infty$), then the Borel asymptotic dimension and the standard asymptotic dimension of $(X, \rho)$ are equal.
\end{thm}

To prove this we need two supporting lemmas. We say that the standard (respectively Borel) $(r, R)$-dimension of $(X, \rho)$ is $d$ if $d \in \N$ is least with the property that there exist sets (respectively Borel sets) $U_0, \ldots, U_d$ covering $X$ such that for every $i$ each class of $\walkeq{U_i}{r}$ has diameter at most $R$, and is $\infty$ if no such $d$ exists.

\begin{lem}\label{lem:smooth_asdim}
Let $(X, \rho)$ be a proper Borel extended metric space having standard $(r, R)$-dimension $d$. If $Y \subseteq X$ is Borel and $\walkeq{Y}{r}$ is bounded, then $Y$ has Borel $(r, R)$-dimension at most $d$.
\end{lem}

\begin{proof}
This lemma is proved using a similar idea to Proposition~\ref{lem:smooth_coloring}.
We assume $d < \infty$, as otherwise there is nothing to prove. Since $\rho$ is proper, $\walkeq{Y}{r}$ is a finite Borel equivalence relation, and each equivalence class admits only finitely many partitions into $d + 1$ pieces. Let $X^{[< \infty]}$ be the standard Borel space of finite subsets of $X$. Let $P$ be the set of $(D_0, \ldots, D_d) \in \prod_{i \leq d} X^{[< \infty]}$ such that $D_0, \ldots, D_d$ are disjoint, and their union $C = D_0 \cup \ldots \cup D_d$ is an equivalence class of $\walkeq{Y}{r}$. Let $P' \subseteq P$ be the set of $(D_0, \ldots, D_d) \in P$ such that for each $i$, the $\walkeq{D_i}{r}$-classes have diameter at most $R$. Since we are assuming $Y$ has $(r,R)$-dimension at most $d$, each $\walkeq{Y}{r}$ class $C$ has at least one partition in $P'$. It is easy to check using Lusin-Novikov uniformization that $P$ and $P'$ are Borel. We want to find a Borel way of taking each $\walkeq{Y}{r}$ class $C$, and choosing exactly one partition of it that lies in $P'$. 

Let $F$ be the equivalence relation on $P'$ where $(D_0, \ldots, D_d) \mathrel{F} (D_0', \ldots, D_d')$ if $\bigcup_{i \leq d} D_i = \bigcup_{i \leq d} D'_i$, so they partition the same $\walkeq{Y}{r}$ class.  Then $F$ is a Borel equivalence relation with finite classes, so we can use a Borel selector to obtain a Borel set $P_0 \subseteq P'$ that contains exactly one partition of each $\walkeq{Y}{r}$ class. The sets $U_i = \{x \colon (\exists (D_0, \ldots, D_d) \in P_0) x \in D_i\}$ are Borel by Lusin-Novikov uniformization, and they witness that $Y$ has Borel $(r,R)$-dimension at most $d$.
\end{proof}

\begin{lem}\label{lem:finiteunion}
Suppose that $(X,\rho)$ is a Borel extended metric space, and that $X$ is the union of Borel sets $X_0 \cup \cdots \cup X_{n-1}$. Assume that each $X_i$ has Borel $(15 r_i, r_{i+1})$-dimension at most $d$, where the $r_i$'s are non-decreasing. Then $(X, \rho)$ has Borel $(r_0, 5 r_n)$-dimension at most $d$.
\end{lem}

\begin{proof}
This is trivial if $d = \infty$, so we assume $d < \infty$. Passing to a subset cannot increase Borel asymptotic dimension, so we may assume that the sets $X_i$ are pairwise disjoint.  By induction on $n$, it suffices to show that if $\{X_0, X_1\}$ is a Borel partition of $X$ where $X_0$ has Borel $(r_0, r_1)$-dimension at most $d$ and $X_1$ has Borel $(3 r_1, r_2)$-dimension at most $d$, then $X$ has Borel $(r_0, 5 r_2)$-dimension at most $d$.

Our assumptions imply that there are Borel sets $U_0^0, \ldots, U_d^0$ covering $X_0$ where the diameter of every class in $\walkeq{U_i^0}{r_0}$ is at most $r_1$, and there are Borel sets $U_0^1, \ldots, U_d^1$ covering $X_1$ where the diameter of every class in $\walkeq{U_i^1}{3r_1}$ is at most $r_2$. Define $U_i = U_i^0 \cup U_i^1$. Then the $U_i$'s are Borel and cover $X$. Notice that the classes of $\walkeq{U_i}{r_0}$ are obtained by gluing together classes of $\walkeq{U_i^0}{r_0}$ with classes of $\walkeq{U_i^1}{r_0}$ whenever the distance between them is at most $r_0$. If $D_0$ is a $\walkeq{U_i^0}{r_0}$-class then, since $D$ has diameter at most $r_1$ and $2 r_0 + r_1 \leq 3 r_1$, the set of points in $U_i^1$ that are distance at most $r_0$ from $D_0$ must lie in a single $\walkeq{U_i^1}{3 r_1}$-class. Therefore, every class of $\walkeq{U_i}{r_0}$ is either a $\walkeq{U_i^0}{r_0}$-class or else it is the union of a subset $B$ of some $\walkeq{U_i^1}{3r_1}$-class and the $\walkeq{U_i^0}{r_0}$-classes that are distance at most $r_0$ to $B$. We conclude that every $\walkeq{U_i}{r_0}$-class has diameter at most $r_2 + 2 (r_0 + r_1) \leq 5 r_2$.
\end{proof}

\begin{proof}[Proof of Theorem \ref{thm:equal}]
Suppose that the standard asymptotic dimension of $(X,\rho)$ is $d$ and that its asymptotic separation index is $s < \infty$. Fix $r_0$; we will find some $R$ such that $(X,\rho)$ has Borel $(r_0,R)$-dimension at most $d$.

As $(X,\rho)$ has standard asymptotic dimension $d$, there is some $r_1 \geq r_0$ such that $X$ has standard $(15r_0, r_1)$-dimension at most $d$.  Continue this process, choosing for each $i \in s+1$ some $r_{i+1} \geq r_i$ such that $X$ has standard $(15r_i, r_{i+1})$-dimension at most $d$.  Finally, using the fact that $(X,\rho)$ has asymptotic separation index $s$, find a Borel partition $\{X_0, X_1, \ldots, X_s\}$ of $X$ such that $\walkeq{X_i}{15 r_s}$ is bounded for every $i \in s+1$. Lemma \ref{lem:smooth_asdim} implies that each $X_i$ has Borel $(15r_i, r_{i+1})$-dimension at most $d$, and thus Lemma \ref{lem:finiteunion} ensures that $(X,\rho)$ has Borel $(r_0, 5r_{s+1})$-dimension at most $d$ as required.
\end{proof}

We next turn our attention to asymptotic separation index and will need an important technique that will recur throughout this paper. Asymptotic dimension is fundamentally concerned with bounding the multiplicities of boundaries at any prescribed large scale (as can be seen most clearly from condition (1) of Lemma \ref{lem:equivalent_asym_dim}). The technique we need therefore relates to ensuring that boundaries stay separated from each other, and is based on two simple observations. The first is that if the classes of $\walkeq{U_1}{2a}$ have diameter at most $b$, then the sets $B(D; a)$, where $D$ varies over all $\walkeq{U_1}{2a}$-classes, are pairwise disjoint and have diameter at most $2a+b$. So for any given set $V_2$, if we union $V_2$ with the sets $B(D; a)$ meeting $V_2$, we obtain a set $V_2 \subseteq U_2 \subseteq B(V_2; 2a+b)$ with the property that for every $x \in U_1$ the set $B(x; a)$ is either disjoint with $U_2$ or contained in $U_2$ (i.e. the boundaries of $U_1$ and $U_2$ are separated by distance at least $a$). The second is that if $U \subseteq B(V; a)$ then every $\walkeq{U}{r}$-class is contained within distance $a$ of some $\walkeq{V}{2a + r}$-class; in particular, if $\walkeq{V}{2a+r}$ is uniformly bounded then so is $\walkeq{U}{r}$. The next lemma formalizes this technique when applied to a sequence $(V_n)_{n \in \N}$.

\begin{lem} \label{lem:buffer}
Let $(X, \rho)$ be a Borel extended metric space with $E_\rho$ countable and for each $n \in \N$ let $V_n \subseteq X$ be Borel and let $a_n, b_n, r_n \geq 0$. Assume for every $n \in \N$ that every $\walkeq{V_n}{6 a_n}$-class has diameter at most $b_n$ and that $a_n \geq r_n + \sum_{k=0}^{n-1} (4 a_k + b_k)$. Then there are Borel sets $U_n \supseteq V_n$ such that $\walkeq{U_n}{r_n}$ is uniformly bounded for every $n$ and whenever $n < m$ and $x \in U_n$ the set $B(x; r_n)$ is either disjoint with or contained in $U_m$.
\end{lem}

\begin{proof}
Notice that every $\walkeq{B(V_n; 2 a_n)}{2 a_n}$-class is contained within distance $2 a_n$ of some $\walkeq{V_n}{6 a_n}$-class and thus has diameter at most $4 a_n + b_n$. For $Y \subseteq X$ define the saturation $s_n(Y) = Y \cup [Y \cap B(V_n; 2 a_n)]_{\walkeq{B(V_n; 2 a_n)}{2 a_n}}$. Then $s_n(Y)$ is Borel when $Y$ is Borel and $s_n(Y) \subseteq B(Y; 4a_n + b_n)$. Set $U_0 = V_0$, for $n \geq 1$ set $U_n = s_0 \circ \cdots \circ s_{n-1}(V_n)$, and notice $U_n \subseteq B(V_n; a_n)$. From this last containment and the inequality $r_n \leq a_n$, we see that $\walkeq{U_n}{r_n}$ is uniformly bounded. Now consider $n < m$ and $x \in U_n$ and assume that $B(x; r_n)$ is not a subset of $U_m$. Since $B(x; a_n) \subseteq B(V_n; 2 a_n)$ is contained in a single $\walkeq{B(V_n; 2 a_n)}{2 a_n}$-class, it follows from the definition of $s_n$ and $U_m$ that $B(x; a_n)$ is disjoint with $W = s_n \circ \cdots \circ s_{m-1}(V_m)$. Since $U_m \subseteq B(W; \sum_{k=0}^{n-1} 4 a_k + b_k)$ and $r_n + \sum_{k=0}^{n-1} 4 a_k + b_k \leq a_n$, it follows that $B(x; r_n)$ is disjoint with $U_m$.
\end{proof}

Below, for sets $A$ and $B$, we say that $A$ \emph{divides} $B$ if $A \cap B$ is a non-empty proper subset of $B$.

\begin{cor} \label{cor:ortho}
Let $(X, \rho)$ be a Borel extended metric space with $E_\rho$ countable and with Borel asymptotic dimension $d < \infty$, and let $(r_n)_{n \in \N}$ be a sequence of radii. Then there exists a sequence of Borel covers $\cU^n = \{U_0^n, \ldots, U_d^n\}$ of $X$ satisfying:
\begin{enumerate}
\item[(i)] $\walkeq{U_i^n}{r_n}$ is uniformly bounded for every $n$ and $i$;
\item[(ii)] if $n < m$, then for every $x \in U_i^n$ the set $B(x; r_n)$ is either disjoint with or a subset of $U_i^m$.
\end{enumerate}
\end{cor}

\begin{proof}
Set $a_0 = r_0$. Given $a_n > 0$, pick Borel sets $V_0^n, \ldots, V_d^n$ covering $X$ with the property that $\walkeq{V_i^n}{6 a_n}$ is uniformly bounded for every $i$, let $b_n > 0$ bound the diameters of every class of each relation $\walkeq{V_i^n}{6 a_n}$, and set $a_{n+1} = r_{n+1} + \sum_{k=0}^n 4 a_k + b_k$. For every $i \in d+1$, apply Lemma \ref{lem:buffer} to $\{V_i^n : n \in \N\}$ to obtain Borel sets $U_i^n \supseteq V_i^n$ satisfying conditions (i) and (ii).
\end{proof}

We will now prove that the asymptotic separation index is often at most $1$ (in fact we know of no examples where it is both finite and greater than one). Our proof will use the following criterion.

\begin{lem} \label{lem:toast}
Let $(X, \rho)$ be a Borel extended metric space with $E_\rho$ countable. Assume that for every $r > 0$ there are Borel sets $V_n \subseteq X$, $n \in \N$, such that $\walkeq{V_n}{3 r}$ is bounded for every $n$, for every $x$ there is $n$ with $B(x; r) \subseteq V_n$, and whenever $n < m$ and $x \in V_n$ the set $B(x; 4r)$ is either contained in or disjoint with $V_m$. Then $\asi(X, \rho) \leq 1$.
\end{lem}

\begin{proof}
Let $r > 0$ and let $(V_n)_{n \in \N}$ be as described. Let $U_1$ be the set of $x \in X$ for which there is an $n$ with $V_n$ dividing $B(x; r)$, and set $U_0 = X \setminus U_1$. For $x \in U_1$ let $n_x$ be such that $V_{n_x}$ divides $B(x; r)$, and for $x \in U_0$ let $n_x$ be least with $B(x; r) \subseteq V_{n_x}$. Clearly if $x, y \in U_0$ and $\rho(x, y) \leq r$ then $n_x = n_y$. Similarly, if $x, y \in U_1$ and $\rho(x, y) \leq r$ then $B(y; 2 r)$ is divided by both $V_{n_x}$ and $V_{n_y}$ and hence $n_y = n_x$. So for $x \in U_i$ we have $[x]_{\walkeq{U_i}{r}}$ is contained in a $\walkeq{B(V_{n_x}; r)}{r}$-class, and has finite diameter since $\walkeq{V_{n_x}}{3r}$ is bounded.
\end{proof}

The above criterion has some rough similarity with the notions of ``toast'' in \cite{GJKSa,GJKSb,GJKSc} and ``barriers'' in \cite{CM16}. In fact, the proof of (b) below is similar to the proof of \cite[Thm. B]{CM16}.

\begin{thm} \label{thm:asi}
Let $(X, \rho)$ be a Borel extended metric space with $E_\rho$ countable.
\begin{enumerate}
\item[(a)] If $\Basdim(X, \rho) < \infty$ then $\asi(X, \rho) \leq 1$.
\item[(b)] Assume $\rho$ is proper and fix a compatible Polish topology on $X$. Then there is an $E_\rho$-invariant comeager set $X' \subseteq X$ such that $\asi(X', \rho) \leq 1$.
\end{enumerate}
\end{thm}

We point out that item (b) and Theorem \ref{thm:equal} together imply Theorem \ref{intro:bp} from the introduction.

\begin{proof}
(a). Set $d = \Basdim(X, \rho)$ and let $r > 0$. Set $r_n = 4r + n$ and apply Corollary \ref{cor:ortho} to obtain a sequence of Borel covers $\cU^n = \{U_0^n, \ldots, U_d^n\}$ satisfying clauses (i) and (ii) of that corollary. For each $x \in X$ define $c(x) \in d+1$ to be least with $x \in U_{c(x)}^n$ for infinitely many $n$. An immediate consequence of (ii) is that $B(x; r_n) \subseteq U_{c(x)}^m$ for every $n$ and infinitely many $m \geq n$. Consequently, $c(x) = c(y)$ when $\rho(x, y) < \infty$. Thus $c$ is Borel and $E_\rho$-invariant. Finally, it is easily seen that the sets $V_n = \bigcup_{i \in d+1} U_i^n \cap c^{-1}(i)$ satisfy the conditions of Lemma \ref{lem:toast}.

(b). Let $r > 0$. Set $a_0 = 4r$ and inductively define $a_{n+1} = 4r + \sum_{k=0}^n (4 a_k + 2r)$. Let $\Gamma_n$ be the Borel graph on $X$ where $(x, y) \in \Gamma_n$ if and only if $0 < \rho(x, y) \leq 6 a_n + 2 r$. Then $\Gamma_n$ is locally finite since $\rho$ is proper, so \cite[Prop. 4.3]{KST99} implies that there is a Borel proper coloring $c_n : X \rightarrow \N$ of $\Gamma_n$. Set $Y_i^n = c_n^{-1}(i)$.

Consider the set $P$ of pairs $(x, s) \in X \times \N^\N$ with the property that for every $x'$ with $\rho(x, x') < \infty$ there is $n$ with $x' \in Y_{s(n)}^n$. By applying the Lusin--Novikov uniformization theorem, we can pick Borel functions $f_n : X \rightarrow X$, $n \in \N$, satisfying $E_\rho = \{(x, f_n(x) : x \in X, \ n \in \N)\}$. Since for every $x \in X$ and $k, n \in \N$ there is $i \in \N$ with $f_k(x) \in Y_i^n$, it is immediate that the set $\{s \in \N^\N : \exists n \in \N \ f_k(x) \in Y_{s(n)}^n\}$ is open and dense. By taking the intersection over $k \in \N$, we find that the set $\{s \in \N^\N : (x, s) \in P\}$ is comeager for every $x \in X$. Since
$$P = \bigcap_{k \in \N} \bigcup_{n \in \N} \{(x,s) \in X \times \N^{\N} : c_n(f_k(x)) = s(n)\}$$
is Borel and in particular has the Baire property, it follows from the Kuratowski--Ulam theorem \cite[Thm. 8.41]{K95} that there is a comeager set of $s \in \N^\N$ such that $\{x \in X : (x, s) \in P\}$ is comeager.

Now fix $s \in \N^\N$ with the property that the set $X' = \{x \in X : (x, s) \in P\}$ is comeager. From the definition of $P$ it is immediate that $X'$ is $E_\rho$-invariant. Set $W_n = B(Y_{s(n)}^n; r) \cap X'$. By construction each class of $\walkeq{W_n}{6 a_n}$ coincides with $B(y; r)$ for some $y \in Y_{s(n)}^n$ and thus has diameter at most $2r$. Therefore we can apply Lemma \ref{lem:buffer}, using $r_n = 4 r$ for all $n$, to obtain Borel sets $V_n \supseteq W_n$ such that $\walkeq{V_n}{3 r}$ is uniformly bounded for every $n$ and whenever $n < m$ and $x \in V_n$ the set $B(x; 4 r)$ is either disjoint with or contained in $V_m$. Additionally, by our choice of $s$ for every $x \in X'$ there is $n$ with $x \in Y_{s(n)}^n$ and hence $B(x; r) \subseteq W_n \subseteq V_n$. Therefore $\asi(X', \rho) \leq 1$ by Lemma \ref{lem:toast}.
\end{proof}

\section{Polynomial growth and bounded packing} \label{sec:polygrow}

In this section we observe the basic fact that Borel actions of finitely-generated groups of polynomial volume-growth have finite Borel asymptotic dimension. We also explore notions of bounded packing, some of which will be significant to the next section.

Our work in the first half of this section, dealing with groups of polynomial volume-growth, essentially replicates the work of Jackson--Kechris--Louveau \cite{JKL02} but is written in the slightly different language of asymptotic dimension.

\begin{defn}
We say that a proper extended metric space $(X, \rho)$ has \emph{bounded packing of degree $d$} if for every $t > 1$ and $r_0 > 0$ there is $r \geq r_0$ so that for all $x \in X$, $B(x; r t)$ contains at most $2 (t+1)^d$-many pairwise disjoint balls of radius $r$.
\end{defn}

\begin{lem} \label{lem:polygrowth}
Let $G$ be a finitely generated group with polynomial volume-growth of degree $d$, and let $B$ be a finite symmetric generating set for $G$ with $1_G \in B$. Then for every $m \geq 1$ and $r_0 > 0$ there is $r \geq r_0$ so that $|B^{r m}| \leq 2 m^d |B^r|$.
\end{lem}

\begin{proof}
Fix $m \geq 1$, $r_0 > 0$, and let $c > 0$ satisfy $|B^r| \leq c r^d$ for all $r \geq 1$. Towards a contradiction, suppose $|B^{r m}| > 2 m^d |B^r|$ for every $r \geq r_0$. Then by induction, for every $k \in \N$
$$c m^{kd} r_0^d \geq |B^{m^k r_0}| > 2^k m^{kd} |B^{r_0}| \geq 2^k m^{kd},$$
which is a contradiction when $c r_0^d < 2^k$.
\end{proof}

Recall from Section \ref{sec:prelim} that if $B$ is any finite generating set for $G$ containing $1_G$ then we obtain a proper right-invariant metric $\tau_B$ on $G$. Furthermore, for an action $G \acts X$ we obtain a proper extended metric $\rho_{\tau_B}$ on $X$.

\begin{cor} \label{cor:polygrowth_pack}
Let $G$ be a finitely-generated group with polynomial volume-growth of degree $d$, let $B$ be a finite symmetric generating set for $G$ with $1_G \in B$, and let $G \acts X$ be any action. Then $(X, \rho_{\tau_B})$ has bounded packing of degree $d$.
\end{cor}

\begin{proof}
Fix $t > 1$ and $r_0 > 0$. Set $m = t+1$ and let $r \geq r_0$ be as given by Lemma \ref{lem:polygrowth}. Fix $x \in X$ and suppose that the balls $B(y_i; r)$, $i \in k$, are pairwise disjoint and contained in $B(x; rt) = B^{r t} \cdot x$. Pick $g_i \in B^{r t}$ with $g_i \cdot x = y_i$. Since the sets $B(y_i; r) = B^r g_i \cdot x$ are pairwise disjoint, the sets $B^r g_i$ must be pairwise disjoint subsets of $B^{r(t+1)} = B^{r m}$. Therefore $k \leq \frac{|B^{r m}|}{|B^r|} \leq 2 m^d = 2 (t+1)^d$.
\end{proof}

\begin{lem} \label{lem:polygrowth_asdim}
Let $(X, \rho)$ be a proper Borel extended metric space with bounded packing of degree $d$. Then $\Basdim(X, \rho) < 2 \cdot 5^d < \infty$.
\end{lem}

\begin{proof}
Since $\rho$ is proper, $E_\rho$ is countable so there are Borel functions $f_n \from X \to X$ with $E_\rho = \{(x, f_n(x)) : x \in X, \ n \in \N\}$. Fix $r_0 > 0$, set $t = 4$, and let $r \geq r_0$ be as in the definition of bounded packing. We will verify condition (1) of Lemma \ref{lem:equivalent_asym_dim}. Consider the Borel graph $\Gamma$ where $x, y \in X$ are joined by an edge if $0 < \rho(x, y) \leq 2r$. Then $\Gamma$ is locally finite since $\rho$ is proper, so by \cite[Prop. 4.2 and Prop. 4.3]{KST99} there is a $\Gamma$-independent Borel set $Y \subseteq X$ that is maximal, meaning every $x \in X \setminus Y$ is adjacent to some $y \in Y$. Therefore the sets $B(y; r)$, $y \in Y$, are pairwise disjoint and $B(Y; 2r) = X$. Now define $s \from X \to Y$ by setting $s(x) = f_k(x)$ if there is $k \in \N$ with $f_k(x) \in Y$ and $\rho(x, f_k(x)) \leq r$, and otherwise setting $s(x) = f_k(x)$ where $k \in \N$ is least with $f_k(x) \in Y$ and $\rho(x, f_k(x)) \leq 2r$. Lastly, define $E$ by $x \mathrel{E} x' \Leftrightarrow s(x) = s(x')$. Then $E$ is Borel and every class has diameter at most $4r$, so $E$ is uniformly bounded. Also, for any $x \in X$ the sets $B(y; r)$, $y \in s(B(x;r))$, are pairwise disjoint and contained in $B(x; 4r)$, so the number of $E$-classes that meet $B(x; r)$ is $|s(B(x;r))| \leq 2 \cdot 5^d$.
\end{proof}

\begin{cor} \label{cor:polygrow_asdim}
If $G$ is a countable group having uniform local polynomial volume-growth of degree $d$, $X$ is a standard Borel space, and $G \acts X$ is any Borel action, then $\Basdim(G \acts X) < 2 \cdot 5^d < \infty$.
\end{cor}

\begin{proof}
This follows from Corollary \ref{cor:subgroups} (which extends to Borel asymptotic dimension, as noted in the final paragraph of Section \ref{sec:asdim}), Corollary \ref{cor:polygrowth_pack}, and Lemma \ref{lem:polygrowth_asdim}.
\end{proof}

We next explore a variant notion of bounded packing that applies to group extensions and will help us to move beyond groups of polynomial volume-growth. The form of the next definition (the use of $B^3$ in particular) is tailored to the specific needs we will encounter in the next section.

\begin{defn} \label{defn:autopack}
Let $G$ be a countable group and let $\Phi$ be a finite set of automorphisms of $G$. We say that $G$ has \emph{$\Phi$-bounded packing} if there is $k$ with the following property: for every finite set $B_0 \subseteq G$ there is a finite symmetric set $B \supseteq B_0 \cup \{1_G\}$ such that $|T| \leq k$ whenever $\phi \in \Phi$, $T \subseteq B^3 \phi(B^3)$ and $(T T^{-1} \setminus \{1_G\}) \cap B = \varnothing$.
\end{defn}

\begin{lem} \label{lem:autopack}
Let $G$ be a countable group and $\Phi$ a finite set of automorphisms of $G$. Assume that either:
\begin{enumerate}
\item[(i)] $G$ has uniform local polynomial volume-growth and $\Phi = \{\mathrm{id}\}$ consists of the trivial automorphism;
\item[(ii)] $G$ is a finite group, an increasing union of finite characteristic subgroups, or a countable torsion-free abelian group with finite $\Q$-rank; or
\item[(iii)] $G = \bigoplus_{n \in \Z} \Z_2$ and each $\phi \in \Phi$ shifts the $\Z$-coordinates of $G$ by some $s \in \Z$.
\end{enumerate}
Then $G$ has $\Phi$-bounded packing.
\end{lem}

\begin{proof}
(i). Say $G$ has uniform local polynomial volume-growth of degree at most $d$ and set $k = 2 \cdot 13^d$. Let $B_0 \subseteq G$ be finite. Without loss of generality we can assume $B_0$ is symmetric and contains $1_G$. Apply Lemma \ref{lem:polygrowth} to obtain $r$ such that $|B_0^{13r}| \leq 2 (13)^d |B_0^r|$, and set $B = B_0^{2 r}$. If $T \subseteq B^6$ satisfies $(T T^{-1} \setminus \{1_G\}) \cap B = \varnothing$ then the sets $B_0^r \cdot t$, $t \in T$, are pairwise disjoint subsets of $B_0^{13r}$ and hence $|T| \leq |B_0^{13r}| / |B_0^r| \leq k$.

(ii). When $G$ is finite this is trivial; simply take $B = G$ and $k = 1$. Next suppose that $G$ is the increasing union of finite characteristic subgroups $G_n$, and set $k = 1$. If $B_0 \subseteq G$ is finite, pick any $n$ with $B_0 \subseteq G_n$ and set $B = G_n$. Since $B = G_n$ is a characteristic subgroup, we have $B^3 \phi(B^3) = G_n$ for all automorphisms $\phi$ of $G$, and any $T \subseteq G_n$ satisfying $(T T^{-1} \setminus \{1_G\}) \cap G_n = \varnothing$ must be a singleton or empty, meaning $|T| \leq 1 = k$.

Now suppose $G$ is a countable torsion-free abelian group with finite $\Q$-rank $d$. Let $A$ be a finite symmetric set that contains $1_G$ and generates a subgroup $G_0$ of $\Q$-rank $d$. Then every element of $G / G_0$ has finite order, and since $G_1 = \langle G_0 \cup \bigcup_{\phi \in \Phi} \phi(G_0) \rangle$ is finitely generated, $G_1 / G_0$ must be finite. Pick a finite set $\Delta$ containing $1_G$ with $G_0 \Delta = G_1$, and pick $\lambda \in \N$ with $(\bigcup_{\phi \in \Phi} \Delta \phi(A) \Delta^{-1}) \cap G_0 \subseteq A^\lambda$. Then for $\phi \in \Phi$, $r \in \N$, and $a_1, \ldots, a_r \in A$, if we pick $\delta_i \in \Delta$ with $\phi(a_1 a_2 \cdots a_i) \in G_0 \delta_i$ we obtain
$$\phi(a_1 \cdots a_r) = (\phi(a_1) \delta_1^{-1}) (\delta_1 \phi(a_2) \delta_2^{-1}) \cdots (\delta_{r-1} \phi(a_r) \delta_r^{-1}) \delta_r \in A^{\lambda r} \Delta.$$
Therefore $\phi(A^r) \subseteq A^{r \lambda} \Delta$ for all $\phi \in \Phi$ and $r \in \N$.

Set $k = 2 (7 + 6 \lambda)^d |\Delta|$ and let $B_0 \subseteq G$ be finite. Notice that since $G$ is a torsion-free abelian group, for every $m \geq 1$ the map $\pi_m : G \rightarrow G$ given by $\pi_m(g) = g^m$ is an injective homomorphism. Since every element of $G / G_0$ has finite order, there is $m \geq 1$ with $\pi_m(B_0) \subseteq G_0$. As finitely generated abelian groups of $\Q$-rank $d$ have polynomial volume-growth of degree at most $d$, we can apply Lemma \ref{lem:polygrowth} to obtain $r \in \N$ with $A^r \supseteq \pi_m(B_0)$ and $|A^{(7 + 6 \lambda)r}| \leq 2 (7 + 6 \lambda)^d |A^r|$. Now set $B = \pi_m^{-1}(A^{2r})$. Then $B$ is finite, symmetric, and contains $B_0 \cup \{1_G\}$. Lastly, consider $\phi \in \Phi$ and suppose $T \subseteq B^3 \phi(B^3)$ satisfies $(T T^{-1} \setminus \{1_G\}) \cap B = \varnothing$. Then the sets $A^r \cdot \pi_m(t)$, $t \in T$, are pairwise disjoint subsets of $A^{7r} \phi(A^{6 r}) \subseteq A^{(7 + 6 \lambda) r} \Delta$. Therefore
$$|T| \leq \frac{|A^{(7 + 6 \lambda) r}|}{|A^r|} \cdot |\Delta| \leq 2 (7 + 6 \lambda)^d |\Delta| = k.$$

(iii). Suppose that every $\phi \in \Phi$ shifts the $\Z$-coordinates by at most $s \in \N$ units. Set $k = 2^{2s}$ and let $B_0 \subseteq G$ be finite. For $m \in \N$ define the subgroup $F_m = \bigoplus_{n = -m}^m \Z_2$. Pick any $m \geq s$ with $B_0 \subseteq F_m$ and set $B = F_m$. Then $B^3 = B$ and $\phi(B) \subseteq F_{m+s}$ for every $\phi \in \Phi$. Let $\phi \in \Phi$ and consider $T \subseteq B^3 \phi(B^3)$ satisfying $(T T^{-1} \setminus \{1_G\}) \cap B = \varnothing$. Then the sets $F_m t$, $t \in T$, are pairwise disjoint subsets of $F_{m+s}$ and hence $|T| \leq |F_{m+s} : F_m| = 2^{2s} = k$.
\end{proof}

The notion of $\Phi$-bounded packing immediately leads to a packing property for group extensions.

\begin{lem} \label{lem:steppack0}
Let $G \lhd H$ be countable groups and let $C \subseteq H$ be finite. Assume that $G$ has $\{\phi_c : c \in C\}$-bounded packing where $\phi_c(g) = c g c^{-1}$. Then there is $k_G$ with the following property: for any finite set $B_0 \subseteq G$ there is a finite symmetric set $B_0 \subseteq B \subseteq G$ with $1_H \in B$ such that $|T| \leq k_G$ whenever $T \subseteq B^3 C B^3$ satisfies $(T T^{-1} \setminus \{1_G\}) \cap B = \varnothing$.	
\end{lem}

\begin{proof}
Set $k_G = k |C|$ where $k$ is as in Definition \ref{defn:autopack}. Let $B_0 \subseteq G$ be finite. By assumption there is a finite symmetric set $B_0 \subseteq B \subseteq G$ with $1_G \in B$ and with the property that $|S| \leq k$ whenever $c \in C$ and $S \subseteq B^3 c B^3 c^{-1}$ satisfies $(S S^{-1} \setminus \{1_G\}) \cap B = \varnothing$. Now let $T \subseteq B^3 C B^3$ satisfy $(T T^{-1} \setminus \{1_G\}) \cap B = \varnothing$. Then for $c \in C$ the set $T_c = (T \cap B^3 c B^3) c^{-1}$ satisfies $T_c \subseteq B^3 c B^3 c^{-1}$ and $(T_c T_c^{-1} \setminus \{1_G\}) \cap B = \varnothing$. Therefore $|T_c| \leq k$ and $|T| \leq k |C| = k_G$.
\end{proof}

For our work with normal series in the next section, we will need the following relativized version of the previous lemma. 

\begin{cor} \label{cor:steppack}
Let $H$ be a countable group, let $F \lhd G$ be normal subgroups of $H$, and let $C \subseteq H$ be finite. Assume that $G / F$ has $\{\phi_c : c \in C\}$-bounded packing where $\phi_c(g F) = c g c^{-1} F$, and let $k_{G/F}$ be as in Lemma \ref{lem:steppack0}. Then for any finite set $B_0 \subseteq G$ there are finite symmetric sets $A \subseteq F$ and $B_0 \subseteq B \subseteq G$ with $1_H \in A \cap B$ such that $|T| \leq k_{G/F}$ whenever $T \subseteq B^3 C B^3$ satisfies $(T T^{-1} \setminus \{1_G\}) \cap A B A = \varnothing$.
\end{cor}

\begin{proof}
It suffices to observe that for any finite set $B \subseteq G$ containing $1_H$ there is a finite symmetric set $A \subseteq F$ containing $1_H$ such that whenever $T \subseteq B^3 C B^3$ and $(T T^{-1} \setminus \{1_H\}) \cap A B A = \varnothing$, we have that $T$ maps injectively into $H / F$ and $(T T^{-1} \setminus \{1_H\}) \cap B F = \varnothing$. Indeed, as long as $A$ contains $F \cap B^4 C B^6 C^{-1} B^3$, we will have that when $t_1, t_2 \in B^3 C B^3$ satisfy $t_1 t_2^{-1} \in B F$ then
\begin{equation*}
t_1 t_2^{-1} \in B F \cap (B^3 C B^6 C^{-1} B^3) \subseteq B (F \cap B^4 C B^6 C^{-1} B^3) \subseteq B A \subseteq A B A.\qedhere
\end{equation*}
\end{proof}

\section{Normal series} \label{sec:series}

Our goal in this section is to prove that if a group $G$ has a suitable normal series $\{1_G\} = G_n \lhd G_{n-1} \lhd \cdots \lhd G_0 = G$ then all free Borel actions of $G$ have finite Borel asymptotic dimension. Our argument will be inductive and will ascend through the normal series. However, even in the simplest situation of an extension by $\Z$, we do not know of any property of $G$-actions that both is preserved by $\Z$-extensions and implies all free Borel $G$-actions have finite Borel asymptotic dimension. Our inductive assumption will therefore need to be forward-looking and will be a statement about both the $k^{\text{th}}$ term $G_k$ and the whole group $G$. The condition we need is specified in the definition below.

For an action $G \acts X$ and $B \subseteq G$, we say that $Y \subseteq X$ is \emph{$B$-independent} if $B \cdot Y \cap Y = \varnothing$.

\begin{defn} \label{defn:expand}
Let $G \lhd H$ be countable groups. We say the pair $(G, H)$ has \emph{good asymptotics} if for every standard Borel space $X$ and every free Borel action $H \acts X$, the restricted action $G \acts X$ has finite Borel asymptotic dimension and for every finite set $C \subseteq H$ there is $\ell \in \N$ so that for every finite set $B \subseteq G$ there is a Borel cover $\{W_0, \ldots, W_\ell\}$ of $X$ consisting of $C B \setminus G$-independent sets.
\end{defn}

Ultimately we will show that when $F \lhd G$ are normal subgroups of $H$, if $(F, H)$ has good asymptotics and $G / F$ has a suitable form then $(G, H)$ has good asymptotics. This will require us to structure $H$-actions in a way that diminishes the dynamics of $F$ and emphasizes the structure of $H / F$. The following lemma achieves this, as we will explain before going into the proof.

\begin{lem} \label{lem:expand}
Let $F \lhd H$ be countable groups with $(F, H)$ having good asymptotics. Let $X$ be a standard Borel space, let $H \acts X$ be a free Borel action, and set $d = \Basdim(F \acts X) < \infty$. Then for any finite sets $C \subseteq H$ and $A \subseteq F$ there is a Borel cover $\{U_i^n : i \in d+1, n \in \ell+1\}$ of $X$ and there are finite sets $C A \cap F \subseteq A_n \subseteq F$ such that:
\begin{enumerate}
\item each set $U_i^n$ is $C A \setminus F$-independent;
\item each equivalence relation $\walkeq{U_i^n}{A}$ is uniformly finite;
\item whenever $Z$ is a $\walkeq{U_i^n}{A_n}$-class, $c \in C$, and $m > n$, the set $c A \cdot Z$ is either disjoint with $U_i^m$ or else contained in a single $\walkeq{U_i^m}{A_m}$-class.
\end{enumerate}
\end{lem}

The most important feature of this lemma is the final property, which is best illustrated in the case that $1_H \in A$ and $C^{-1} C \cap F \subseteq A$. Specifically, since $C \cap c F \subseteq c A$, for any $\walkeq{U_i^n}{A_n}$-class $Z$ and any $\walkeq{U_i^m}{A_m}$-class $Z'$ with $m > n$, whether $c \cdot Z$ is contained in or disjoint with $Z'$ only depends on the coset $c F \in C F / F$. Moreover, $C \cdot Z$ is disjoint with $U_i^n \setminus Z$ since $U_i^n$ is $C \setminus F$-independent and $C \cap F \subseteq A_n$. Thus in a certain local (only considering transformations by the elements of $C$) and ordered (requiring $n \leq m$) sense, we are able to view $X$ in a way that allows dynamical data to descend to the quotient $H / F$.

The basic idea of the proof is that we will take a Borel cover $\{W_0, \ldots, W_\ell\}$ of $X$ consisting of $(C^{-1} C A') \setminus F$-independent sets for some carefully chosen $A' \subseteq F$. Using the fact that $\Basdim(F \acts X) = d$, we will choose a Borel cover $\{V_0^n, \ldots, V_d^n\}$ of $W_n$ for each $n \in \ell+1$. Moreover, we will arrange for $\walkeq{V_i^n}{Q_n}$ to be uniformly finite, where $Q_n \supseteq A_n$ is chosen to be large relative to the shapes of the $\walkeq{V_j^k}{Q_k}$-classes for $k < n$, $j \in d+1$. In other words, the $Q_n$'s signify a rapid growth in scales. We will then slightly enlarge each set $V_i^n \cap W_n$ to $U_i^n \subseteq A_n \cdot (V_i^n \cap W_n)$ in a process quite similar to the proofs of Lemma \ref{lem:buffer} and Corollary \ref{cor:ortho}. Specifically, whenever $D$ is a $\walkeq{V_i^n \cap W_n}{Q_n}$-class, $c \in C$, and $A_n^{-1} (C \cap c F) A_n \cdot D$ comes close to $V_i^m \cap W_m$ with $m > n$, we will arrange that $U_i^m$ contain $c A_n \cdot D$. For this enlargement procedure to be stable, we need the sets $A_n^{-1} (C \cap c F) A_n \cdot D$ to be pairwise disjoint. Such disjointness will be derived from the independence properties of the $W_n$'s.

\begin{proof}
Without loss of generality, we may assume $1_H \in A \cap C$. Let $\ell$ be as in Definition \ref{defn:expand} for the set $C^{-1} C$. Set $A_0 = A \cup (C A \cap F)$. Inductively assume that $A_n \subseteq F$ has been defined. Set
$$T_n = A_n^{-1} C A_n \text{ and } Q_n = A_n^{-1} A A_n A^{-1} A_n \cup (T_n^{-1} T_n \cap F) \subseteq F.$$
Let $\{V_0^n, \ldots, V_d^n\}$ be a Borel cover of $X$ such that $\walkeq{V_i^n}{Q_n}$ is uniformly finite for every $i$. Now pick a finite set $A_{n+1} \subseteq F$ containing $A_n$ and satisfying $A_n A_n^{-1} (C \cap c F) A_n \cdot Y \subseteq A_{n+1} \cdot x$ for every $\walkeq{V_i^n}{Q_n}$-class $Y$, $c \in C$, and $x \in A_n^{-1} c A_n \cdot Y$.
	
By the normality of $F$ we have that $T_n^{-1} T_n \subseteq C^{-1} C F$ for every $n$, so by our choice of $\ell$ there is a Borel cover $\{W_0, \ldots, W_\ell\}$ of $X$ so that $W_n$ is $T_n^{-1} T_n \setminus F$-independent for every $n$. The basic idea now will be to define $U_i^n$ to be a saturation (i.e. enlargement) of $V_i^n \cap W_n$, where the saturation operation is designed to achieve property (3).

Before proceeding, we note that if $Y$ and $Y'$ are $\walkeq{V_i^n \cap W_n}{Q_n}$-classes and $c, c' \in C$ satisfy $(c F, Y) \neq (c' F, Y')$ then $A_n^{-1} (C \cap c F) A_n \cdot Y$ and $A_n^{-1} (C \cap c' F) A_n \cdot Y'$ are disjoint. Indeed, notice that the set
$$(A_n^{-1} (C \cap c F) A_n)^{-1} A_n^{-1} (C \cap c' F) A_n$$
is a subset of $T_n^{-1} T_n$ and that moreover it is a subset of $T_n^{-1} T_n \cap F \subseteq Q_n$ when $c F = c' F$ and is a subset of $T_n^{-1} T_n \setminus F$ when $c F \neq c' F$. Therefore when $c F = c' F$ the claim holds since $Y$ and $Y'$ are distinct classes of $\walkeq{V_i^n \cap W_n}{Q_n}$, and when $c F \neq c' F$ the claim holds since $W_n \supseteq Y \cup Y'$ is $T_n^{-1} T_n \setminus F$-independent.

We now define the saturation functions. For $X' \subseteq X$ define $s_i^n(X')$ to be the union of $X'$ with the members of the collection
$$\{A_n^{-1} (C \cap c F) A_n \cdot Y : c F \in C F / F, \ Y \text{ a } \walkeq{V_i^n \cap W_n}{Q_n}\text{-class}\}$$
that have nonempty intersection with $X'$. By the previous paragraph, the collection above consists of pairwise disjoint sets and thus $s_i^n \circ s_i^n(X') = s_i^n(X')$ for all $X' \subseteq X$. Also notice that it is immediate from the definition of $A_{n+1}$ that $A_n \cdot (s_i^n(X') \setminus X') \subseteq A_{n+1} \cdot X'$. Since also $A_n \subseteq A_{n+1}$, we conclude that $A_n \cdot s_i^n(X') \subseteq A_{n+1} \cdot X'$ for all $X' \subseteq X$.

We now verify that $s_i^n(X')$ is Borel whenever $X'$ is Borel. For $x \in X$, we have that $x \in s_i^n(X')$ if and only if $x \in X'$ or there exists $x' \in X'$, $c \in C$, and a $\walkeq{V_i^n \cap W_n}{Q_n}$-class $Y$ such that $x, x' \in A_n^{-1} (C \cap c F) A_n \cdot Y$. More simply, we can avoid reference to $Y$ by requiring that there exist $t, t' \in A_n^{-1} (C \cap c F)^{-1} A_n$ such that $(t \cdot x, t' \cdot x') \in \walkeq{V_i^n \cap W_n}{Q_n}$. Additionally, the definition of $A_{n+1}$ ensures that $x' \in A_{n+1} \cdot x$. Therefore $x \in s_i^n(X')$ if and only if $x \in X'$ or there exist $a \in A_{n+1}$ and $t, t' \in A_n^{-1} (C \cap c F)^{-1} A_n$ such that $a \cdot x \in X'$ and $(t \cdot x, t' a \cdot x) \in \walkeq{V_i^n \cap W_n}{Q_n}$. This shows that $s_i^n(X')$ is Borel when $X'$ is Borel, as claimed. 

Define the Borel sets $U_i^0 = V_i^0 \cap W_0$ and $U_i^n = s_i^0 \circ \cdots \circ s_i^{n-1}(V_i^n \cap W_n)$ for $n \geq 1$. Then $V_i^n \cap W_n \subseteq U_i^n$ so the sets $U_i^n$ cover $X$. Also, since $1_H \in A \subseteq A_0$ and since it always holds that $A_n \cdot s_i^n(X') \subseteq A_{n+1} \cdot X'$, by induction we have that
$$U_i^n \subseteq A \cdot U_i^n \subseteq A_n \cdot (V_i^n \cap W_n).$$
We proceed to check conditions (1) through (3).

(1). Since $W_n$ is $T_n^{-1} T_n \setminus F$-independent and $A_n^{-1} (C \setminus F) A_n = T_n \setminus F \subseteq T_n^{-1} T_n \setminus F$, we have that $A_n \cdot W_n \cap (C \setminus F) A_n \cdot W_n = \varnothing$. Since also $U_i^n \subseteq A \cdot U_i^n \subseteq A_n \cdot W_n$, it follows that $U_i^n \cap (C \setminus F) A \cdot U_i^n = \varnothing$. Therefore $U_i^n$ is $C A \setminus F$-independent as claimed since $C A \setminus F = (C \setminus F) A$.

(2). Since $U_i^n \subseteq A_n \cdot V_i^n$, every $\walkeq{U_i^n}{A}$-class is contained within the $A_n$-translates of some $\walkeq{V_i^n}{A_n^{-1} A A_n}$-class. Since $A_n^{-1} A A_n \subseteq Q_n$, this latter equivalence relation is a subrelation of $\walkeq{V_i^n}{Q_n}$ and is therefore uniformly finite. We conclude that $\walkeq{U_i^n}{A}$ is uniformly finite.

(3). Let $Z$ be a $\walkeq{U_i^n}{A_n}$-class, $c \in C$, and $m > n$. Since $A \cdot Z \subseteq A \cdot U_i^n \subseteq A_n \cdot (V_i^n \cap W_n)$, if we set
$$Y_0 = (A_n^{-1} A \cdot Z) \cap V_i^n \cap W_n$$
then we have $Z \subseteq A \cdot Z \subseteq A_n \cdot Y_0$ and $Y_0 \subseteq V_i^n \cap W_n$. Since $Z$ is a single $\walkeq{U_i^n}{A_n}$-class, every point in $Y_0$ can be translated by an element of $A^{-1} A_n$ to arrive at a point in $Z$, every point in $Z$ can be translated by an element of $A_n^{-1} A$ to arrive at a point in $Y_0$, and since $A_n^{-1} A A_n A^{-1} A_n \subseteq Q_n$, we must have that $Y_0$ is contained in a single $\walkeq{V_i^n \cap W_n}{Q_n}$-class $Y$. Then we have
$$A_n^{-1} c A \cdot Z \subseteq A_n^{-1} c A_n \cdot Y_0 \subseteq A_n^{-1} (C \cap c  F) A_n \cdot Y.$$
Recalling our saturation operation, we see that the set $s_i^n \circ \cdots s_i^{m-1}(V_i^m \cap W_m)$ either contains $A_n^{-1} (C \cap c F) A_n \cdot Y$ or is disjoint with it. Therefore, since
$$s_i^n \circ \cdots \circ s_i^{m-1}(V_i^m \cap W_m) \subseteq U_i^m \subseteq A_n \cdot s_i^n \circ \cdots \circ s_i^{m-1}(V_i^m \cap W_m),$$
it follows that $U_i^m$ either contains $A_n^{-1} (C \cap c F) A_n \cdot Y$ or is disjoint with $(C \cap c F) A_n \cdot Y$. In the latter case $c A \cdot Z$ is disjoint with $U_i^m$. In the former case, $c A \cdot Z$ is contained in $A_n^{-1}(C \cap c F) A_n \cdot Y \subseteq U_i^m$, and the definition of $A_{n+1}$ and the inclusion $A_{n+1} \subseteq A_m$ ensure that $A_n^{-1}(C \cap c F) A_n \cdot Y$ is contained in a single $\walkeq{U_i^m}{A_m}$-class.
\end{proof}

The next lemma provides the crucial inductive step that will allow us to ascend through a normal series.

\begin{lem} \label{lem:polystep}
Let $H$ be a countable group and let $F \lhd G$ be normal subgroups of $H$. Assume that $(F, H)$ has good asymptotics and that $G / F$ has $\{\phi_c : c \in C\}$-bounded packing whenever $C \subseteq H$ is a finite set with the property that each automorphism $\phi_c(g F) = c g c^{-1} F$ is either trivial or an outer automorphism. Then $(G, H)$ has good asymptotics.
\end{lem}

\begin{proof}
Let $X$ be a standard Borel space, let $H \acts X$ be a free Borel action, let $C \subseteq H$ be a finite symmetric set containing $1_H$, and set $d_F = \Basdim(F \acts X) < \infty$. From Definition \ref{defn:expand} we see that without loss of generality we can replace each $c \in C$ with any element of $c G$. We can therefore assume that the conjugation of each $c \in C$ on $G / F$ either is an outer automorphism or is trivial, while still maintaining that $C$ is symmetric and contains $1_H$. Our assumptions then allow us to apply Corollary \ref{cor:steppack} to obtain $k_{G/F} \in \N$. Set $\ell_G = (k_{G/F}+1)(d_F+1)$, fix a finite set $B_0 \subseteq G$, and let $A \subseteq F$ and $B_0 \subseteq B \subseteq G$ be finite symmetric sets containing $1_H$ as described in Corollary \ref{cor:steppack}.

Apply Lemma \ref{lem:expand} to obtain finite sets $(A B^3 A) \cap F \subseteq A_n \subseteq F$ and a Borel cover $\{U_i^n : i \in d_F + 1, \ n \in \ell_F + 1\}$ of $X$ such that each set $U_i^n$ is $(A B A \cup B^3 C B^3) \setminus F$-independent, each relation $\walkeq{U_i^n}{B^3 \cap F}$ is uniformly finite, and whenever $Z$ is a $\walkeq{U_i^n}{A_n}$-class, $t \in B^3 C B^3$, and $m > n$, the set $t ((B A)^2 \cap F) \cdot Z$ is either disjoint with $U_i^m$ or else contained in a single $\walkeq{U_i^m}{A_m}$-class. Next define Borel sets $V_i^n \subseteq U_i^n$ as follows: set $V_i^{\ell_F} = U_i^{\ell_F}$ and by reverse induction on $n \in \ell_F+1$ define
$$V_i^n = U_i^n \setminus \textstyle{ \bigcup_{m > n} B \cdot V_i^m }.$$
Notice that since $1_H \in B$ the collection $\{B \cdot V_i^n : i \in d_F+1, \ n \in \ell_F+1\}$ covers each $U_i^n$ and thus covers $X$. Additionally, it is simple to see by reverse induction on $n$ that $V_i^n$ is a $\walkeq{U_i^n}{A_n}$-invariant subset of $U_i^n$. Consequently, the following properties hold:
\begin{enumerate}
\item[(i)] if $Z$ is a $\walkeq{V_i^n}{A_n}$-class then for every $m > n$ and $t \in B^3 C B^3$ we have that $t \cdot Z$ is either disjoint with $V_i^m$ or else contained in a $\walkeq{V_i^m}{A_m}$-class;
\item[(ii)] if $Z$ and $Z'$ are classes of $\walkeq{U_i^n}{A_n}$ and $\walkeq{U_i^m}{A_m}$ respectively and $Z \neq Z'$, then $A B A \cdot Z \cap Z' = \varnothing$.
\end{enumerate}
Statement (i) holds since each $V_i^n$ is $\walkeq{U_i^n}{A_n}$-invariant, and when $n = m$ statement (ii) holds since $U_i^n \supseteq V_i^n$ is $(A B A) \setminus F$-independent and $(A B A) \cap F \subseteq A_n$. Now suppose $Z$ and $Z'$ are as described in (ii) with $n \neq m$. Since $A B A = (A B A)^{-1}$, we can assume $m > n$. Then for every $b \in B$ we have that $A b A \cdot Z = b (b^{-1} A b) A \cdot Z \subseteq b ((B A)^2 \cap F) \cdot Z$ is either disjoint with or contained in $Z'$, but the containment $b \cdot Z \subseteq A b A \cdot Z \subseteq Z'$ is prohibited by definition of $V_i^n$.

Define a directed Borel graph $\Gamma_i$ on $\bigsqcup_{n \in \ell_F+1} V_i^n$ by placing an edge directed from $x$ to $y$ if $x \in B^3 C B^3 \cdot y$ (equivalently $y \in B^3 C B^3 \cdot x$) and there are $n < m$ with $x \in V_i^n$ and $y \in V_i^m$. We construct a proper Borel coloring $\pi_i \from \bigsqcup_{n \in \ell_F+1} V_i^n \to k_{G/F} + 1$ of $\Gamma_i$ as follows. First let $\pi_i$ have constant value $0$ on $V_i^{\ell_F}$. Now inductively assume that $\pi_i$ is defined on $V_i^m$ and is $\walkeq{V_i^m}{A_m}$-invariant for every $m > n$. Given any $x \in V_i^n$, let $T \subseteq B^3 C B^3$ be such that for each $m > n$ and $\walkeq{V_i^m}{A_m}$-class $Z$ receiving an edge from $x$, there is exactly one $t \in T$ with $t \cdot x \in Z$. Then (ii) and Corollary \ref{cor:steppack} imply that $|T| \leq k_{G/F}$, and our inductive hypothesis implies that the number of $\pi_i$-values used by forward-neighbors of $x$ is at most $|T|$. Moreover, (i) and our inductive hypothesis imply that for all $(x, y) \in \walkeq{V_i^n}{A_n}$ and $t \in B^3 C B^3$, either $\pi_i(t \cdot x) = \pi_i(t \cdot y)$ or else neither is defined. Thus if for each $x \in V_i^n$ we define $\pi_i(x)$ be the least element of $k_{G/F} + 1$ not used by any of the forward-neighbors of $x$, then $\pi_i \res V_i^n$ is Borel and $\walkeq{V_i^n}{A_n}$-invariant.

To finish the proof, set $W_i^j = B \cdot \pi_i^{-1}(j)$. Then we have $\ell_G$-many Borel sets $W_i^j$ that cover $X$. Notice that every $w \in W_i^j$ is a $B$-translate of some $v \in \bigsqcup_{n \in \ell_F+1} V_i^n$ with $\pi_i(v) = j$. Therefore, since $\pi_i$ is a proper coloring, whenever $w, w' \in W_i^j$ satisfy $w \in B^2 C B^2 \cdot w'$, there must be a single $n$ with $w, w' \in B \cdot V_i^n$. It follows that $\walkeq{W_i^j}{B}$ is uniformly finite since each set $U_i^n \supseteq V_i^n$ is $B^3 \setminus F$-independent and each relation $\walkeq{U_i^n}{B^3 \cap F} \supseteq \walkeq{V_i^n}{B^3 \cap F}$ is uniformly finite. As $\ell_G$ was specified before $B$, we find $\Basdim(G \acts X) \leq \ell_G-1 < \infty$.

Finally, we need to check that each set $W_i^j$ is $C B \setminus G$-independent. Suppose $w, w' \in W_i^j$, $h \in C B$, and $h \cdot w = w'$. Since $h \in C B \subseteq B^2 C B^2$, our remark in the previous paragraph implies that there is $n$ and $v, v' \in V_i^n \subseteq U_i^n$ with $w \in B \cdot v$ and $w' \in B \cdot v'$. Consequently, $v' \in B h B \cdot v \subseteq B C B^2 \cdot v$. Since $U_i^n$ is $(B C B^2) \setminus F$-independent, we must have that $B h B \cap F \neq \varnothing$ and hence $h \in B F B \subseteq G$. We conclude that $W_i^j$ is $C B \setminus G$-independent as claimed, and that $(G, H)$ has good asymptotics.
\end{proof}

\begin{thm} \label{thm:normal_asdim}
Let $G$ be a countable group admitting a normal series where each quotient of consecutive terms is a finite group, an increasing union of finite characteristic subgroups, or a torsion-free abelian group with finite $\Q$-rank, except the top quotient can be any group of uniform local polynomial volume-growth or the lamplighter group $\Z_2 \wr \Z$. Then for every free Borel action $G \acts X$ on a standard Borel space $X$ we have $\Basdim(G \acts X) = \asdim(G) < \infty$.
\end{thm}

We remark that the quotients of consecutive terms can be allowed to have a more general form. Specifically, by refining the normal series it is enough to assume that each quotient of consecutive terms, aside from the top quotient, admits a characteristic series where each consecutive quotient is finite, an increasing union of finite characteristic subgroups, or torsion-free abelian with finite $\Q$-rank. For instance, abelian groups that have a finite torsion subgroup and finite $\Q$-rank admit such a characteristic series, as do all finitely generated nilpotent groups, all polycyclic groups, and all solvable groups having finite Pr{\"u}fer rank (as shown in the proof of Lem. \ref{lem:prufer}), so one can allow these groups to appear as consecutive quotients.

\begin{proof}
Let the normal series be $\{1_G\} = G_n \lhd G_{n-1} \lhd \cdots \lhd G_0 = G$. In the case where the top quotient is the lamplighter group, we may enlarge the normal series by one term and re-index so that $G_0 / G_2$ is isomorphic to $\Z_2 \wr \Z$ via an isomorphism identifying $G_1 / G_2$ with $\bigoplus_{n \in \Z} \Z_2$. Since $\asdim(G \acts X) = \asdim(G)$ for all free actions, by Theorem \ref{thm:equal} we only need to show $\Basdim(G \acts X) < \infty$ for all free Borel actions. So it suffices to prove that $(G_k, G)$ has good asymptotics for $k = n, n-1, \ldots, 0$. For $(G_n, G) = (\{1_G\}, G)$ this is trivial, since for any free Borel action $G \acts X$ we clearly have $\Basdim(\{1_G\} \acts X) = 0$ and for any finite $C \subseteq G$ we can apply \cite[Prop. 4.6]{KST99} to partition $X$ into $(2|C|+1)$-many Borel sets that are $C \setminus \{1_G\}$-independent. For the inductive step, suppose $(G_{k+1}, G)$ has good asymptotics for some $n > k \geq 0$. If $G_k / G_{k+1}$ is finite, an increasing union of finite characteristic subgroups, or is a torsion-free abelian group with finite $\Q$-rank, then Lemmas \ref{lem:autopack}.(ii) and \ref{lem:polystep} imply that $(G_k, G)$ has good asymptotics (note this case applies when $G_0 / G_2$ is the lamplighter group and $k = 0$). If $k = 1$ and $G_0 / G_2$ is the lamplighter group, then every outer automorphism of $G_1 / G_2$ induced by $G$ is of the form described in Lemma \ref{lem:autopack}.(iii) and thus that lemma and Lemma \ref{lem:polystep} imply that $(G_1, G)$ has good asymptotics. Finally, if $k = 0$ and $G_0 / G_1$ has uniform local polynomial volume-growth then $G$ induces no outer automorphisms of $G_0 / G_1$ and thus Lemmas \ref{lem:autopack}.(i) and \ref{lem:polystep} imply that $(G_0, G)$ has good asymptotics.
\end{proof}

The above theorem applies to the lamplighter group $\Z_2 \wr \Z$, all groups of uniform local polynomial volume-growth, and all virtually solvable groups having finite Pr{\"u}fer rank. This latter class of groups contains all polycyclic groups and the Baumslag--Solitar group $BS(1, 2)$ (this group has finite Pr{\"u}fer rank but can also be directly realized as an extension of the dyadic rationals, which is a torsion-free abelian group of $\Q$-rank $1$, by $\Z$).

If in Theorem \ref{thm:normal_asdim} we additionally allowed consecutive quotients to be locally finite, then to the best of our knowledge this theorem would cover all countable amenable groups currently known to have finite standard asymptotic dimension. Since Lemma \ref{lem:autopack} fails for many locally finite groups, a different method would be needed to prove this stronger version. On the other hand, Theorem \ref{thm:normal_asdim} cannot be extended to any non-amenable group as this would contradict Theorem \ref{thm:asdimhyp} since such groups admit free actions whose orbit equivalence relation is not hyperfinite \cite{JKL02}.

Another way of improving the above theorem would be to extend it to hold for all (not necessarily free) Borel actions. Hypothetically this might be possible while keeping most of our methods intact, but we were not able to find any way of doing this.

\section{Hyperfinite equivalence relations} \label{sec:hyperfinite}

Our initial motivation for this present work was to study hyperfinite equivalence relations with the goal of providing a positive answer to Weiss' question in the case of polycyclic groups. It was through working towards that goal that the notion of Borel asymptotic dimension slowly but naturally emerged. Perhaps one of the most surprising aspects of this work is the discovery of just how well-suited Borel asymptotic dimension is to the study of hyperfinite equivalence relations, and how in hindsight it may have played an important but hidden role in past investigations into Weiss' question.

We first present the simple argument that $E_\rho$ is hyperfinite whenever $\rho$ is proper and $(X, \rho)$ has finite Borel asymptotic dimension.

\begin{thm} \label{thm:asdimhyp}
Suppose that $(X,\rho)$ is a proper Borel extended metric space, and set $E_\rho = \{(x, x') \in X \times X : \rho(x, x') < \infty\}$. If $(X, \rho)$ has finite Borel asymptotic dimension then $E_\rho$ is hyperfinite.
\end{thm}

\begin{proof}
We will rely on condition (1) of Lemma \ref{lem:equivalent_asym_dim}. Let $d < \infty$ be the Borel asymptotic dimension of $(X, \rho)$. We will inductively build radii $r_n \geq n$ and an increasing sequence of uniformly bounded Borel equivalence relations $F_n$ with the property that for every $n$ and $x \in X$, $B(x; r_n)$ meets at most $d+1$ classes of $F_n$. To begin set $r_0 = 0$ and let $F_0$ be the equality relation on $X$. Inductively assume that $F_{n-1}$ has been constructed. Then $F_{n-1}$ is finite since $\rho$ is proper. So there is a Borel selector $s \from X \rightarrow X$ satisfying $s(y) = s(x) \ F_{n-1} \ x$ for all $(x, y) \in F_{n-1}$. Choose $r_n \geq n$ so that $[x]_{F_{n-1}} \subseteq B(x; r_n)$ for all $x \in X$, and choose a uniformly bounded Borel equivalence relation $E$ with the property that for every $x$, $B(x; 2 r_n)$ meets at most $d+1$ classes of $E$. Now define $F_n = \{(x, y) \in X \times X : s(x) \ E \ s(y)\}$. Then $F_n$ is Borel, uniformly bounded, and contains $F_{n-1}$. Additionally, for any $x$, the number of $F_n$ classes meeting $B(x; r_n)$ is bounded by the number of $E$-classes meeting $B(x; 2 r_n) \supseteq s(B(x; r_n))$, which is at most $d+1$. Now set $F_\infty = \bigcup_n F_n$. Then $F_\infty$ is hyperfinite. For every $n$ and $x$, $B(x; n)$ meets at most $d+1$ classes of $F_n$ and hence at most $d+1$ classes of $F_\infty$. Thus every $E_\rho$-class divides into at most $d+1$ classes of $F_\infty$, which implies $E_\rho$ is hyperfinite as well \cite[Prop. 1.3.(vii)]{JKL02}.
\end{proof}

The above proof together with our proof of Corollary \ref{cor:polygrow_asdim} is essentially identical to the argument of Jackson--Kechris--Louveau that Borel actions of groups of polynomial volume-growth generate hyperfinite equivalence relations \cite{JKL02}. It is natural to wonder then why the utility of Borel asymptotic dimension was not identified earlier. We suspect that part of the reason is that the above argument relies on condition (1) of Lemma \ref{lem:equivalent_asym_dim} while the discoveries we make here depend heavily on condition (2) of that lemma, and the implication (1) $\Rightarrow$ (2) is not immediately obvious.

When we consider the overlap between Weiss' question and the Union Problem, the significance of Borel asymptotic dimension becomes much more apparent. Special cases of this overlap were solved in the works of Gao--Jackson \cite{GJ15} and Schneider--Seward \cite{SS} through highly technical arguments that spanned dozens of pages. However, in the next theorem we will see that a much more general result can be obtained using Borel asymptotic dimension and that the proof is less technical and tremendously shorter in comparison. Furthermore, we believe that our proof here preserves the core concept of those prior works, as we now discuss.

The key concept in \cite{GJ15} and \cite{SS} is the notion of orthogonality. In those works, orthogonality is a property for sequences of finite Borel equivalence relations $E_n$ that are contained in a fixed orbit equivalence relation $E_G^X$, and the property functions to enforce that the boundaries of the $E_n$'s within $E_G^X$ vanish as $n \rightarrow \infty$. In this situation, $E_G^X$ will be seen to be hyperfinite as it will be equal to the union of the increasing sequence of equivalence relations $\bigcap_{k \geq n} E_k$. Roughly speaking, orthogonality was defined by using the geometry of the group $G$ to define various types of facial boundaries and then requiring that every facial boundary of $E_n$ be far (in a manner determined by $n$) from the corresponding facial boundary of $E_m$ when $m > n$. For instance, when $G = \Z^d$ and $e_1, \ldots, e_d$ are the standard generators of $\Z^d$, the type $i$ boundary of $E_n$ is the set of $x$ with $(x, e_i \cdot x) \not\in E_n$ or $(x, -e_i \cdot x) \not\in E_n$. 

We believe that the essence of orthogonality can be captured in the language of asymptotic dimension. Instead of seeking separation of various facial boundary types in order to guarantee the vanishing of boundaries in the limit, one can consider sequences of $(d+1)$-indexed covers and require that for each index $i$ the boundaries of the corresponding sets are separated from each other. In this sense, we already constructed orthogonal sequences in Corollary \ref{cor:ortho}, and this construction was far less technical than those in \cite{GJ15,SS}. Before proceeding, we record a simplified version of that corollary.

\begin{cor} \label{cor:ortho2}
Let $(X, \rho)$ be a Borel extended metric space with $E_\rho$ countable and with Borel asymptotic dimension $d < \infty$. Then for any $r > 0$ and $D \in \N$ there exist Borel covers $\mathcal{V}^j = \{V_0^j, \ldots, V_d^j\}$ of $X$ for $j \in D+1$ such that $\walkeq{V_i^j}{r}$ is uniformly bounded for all $i$ and $j$, and for every $x \in X$ and $i \in d+1$ there is at most one $j \in D+1$ such that $V_i^j$ divides $B(x; r)$.
\end{cor}

\begin{proof}
By using $r_j = 2r$, such a sequence of covers is immediately obtained from Corollary \ref{cor:ortho}.
\end{proof}

In \cite{GJ15,SS} the notion of orthogonality (using facial boundary types as discussed above) was defined for each of the finitely generated subgroups of $G$ and, in confronting the Union Problem, the ultimate technical challenge was to combine these notions together by diagonalizing over all the finitely generated subgroups of $G$. This involved an intricate and cumbersome inductive argument that was choreographed through a $2$-dimensional triangular array. In contrast, with Borel asymptotic dimension we are able to perform this diagonalization in a different way that is much more automated and direct. We put all of these pieces together in the theorem below.

\begin{thm} \label{thm:union}
Let $X$ be a standard Borel space and let $(\rho_n)_{n \in \N}$ be a sequence of proper Borel extended metrics on $X$. Assume for every $n \in \N$ and $r > 0$ that $\sup \{\rho_{n+1}(x, y) : x, y \in X \ \rho_n(x, y) < r\}$ is finite and set $E = \bigcup_n E_{\rho_n}$. If $(X, \rho_n)$ has finite Borel asymptotic dimension for every $n \in \N$ then $E$ is hyperfinite.
\end{thm}

\begin{proof}
Inductively for $n \in \N$, pick $r_n \geq n$ large enough that $B_{\rho_k}(x; n + r_k) \subseteq B_{\rho_n}(x; r_n)$ for all $k < n$ and $x \in X$. So we have $(x, y) \in E \Leftrightarrow \exists n \ \rho_n(x, y) \leq r_n$. Also let $d_n$ be the Borel asymptotic dimension of $(X, \rho_n)$. Inductively we will build finite Borel covers $(\cU^n)_{n \in \N}$, $(\cV^{n,j})_{n \in \N, j \in d_{n+1}+1}$ of $X$ satisfying the following:
\begin{enumerate}
\item[(i)] $\walkeq[\rho_n]{U}{2 r_n}$ is $\rho_n$-uniformly bounded for every $U \in \cU^n \cup \bigcup_{j \in d_{n+1}+1} \cV^{n,j}$;
\item[(ii)] for every $k \in \N$, $n \geq k$, and $x \in X$ there is at most $(d_k+1)$-many covers $\cW \in \{\cU^m : k \leq m \leq n\} \cup \{\cV^{n,j} : j \in d_{n+1}+1\}$ that divide $B_{\rho_k}(x; r_k)$ (meaning there is $W \in \cW$ that divides $B_{\rho_k}(x; r_k)$).
\end{enumerate}
The role of the $\cV^{n,j}$'s will be transient and only used to inductively build the $\cU^n$'s. Denote by $F_n$ the Borel equivalence relation where $(x, x') \in F_n$ if and only if for every $U \in \mathcal{U}^n$ we have either $x, x' \not\in U$ or $(x, x') \in \walkeq[\rho_n]{U}{2 r_n}$. Then (i) implies that $F_n$ is finite. If $(x, y) \in E$ then by construction there is $k$ with $\rho_k(x, y) \leq r_k$, and so from (ii) we see that $(x, y) \in F_n$ for all but at most $(d_k+1)$-many $n \geq k$. Thus, once these covers are constructed we will have that $E = \bigcup_k \bigcap_{n \geq k} F_n$ is hyperfinite. To begin the construction, apply Corollary \ref{cor:ortho2} to obtain Borel covers $\cU^0$, $(\cV^{0,j})_{j \in d_1+1}$ satisfying (i) and (ii) with $k = n = 0$.

Now suppose that $\cU^0, \ldots, \cU^{n-1}$, $\cV^{n-1,0}, \ldots, \cV^{n-1,d_n}$ have been constructed. By enlarging $r_n$ if necessary, we can assume that the $\rho_n$-diameter of the classes of $\walkeq[\rho_{n-1}]{V}{2 r_{n-1}}$ are bounded by $r_n$ for every $j$ and $V \in \cV^{n-1,j}$. Apply Corollary \ref{cor:ortho2} to obtain a sequence of Borel covers $\bar{\cV}^{n,j} = \{\bar{V}_0^{n,j}, \ldots, \bar{V}_{d_n}^{n,j}\}$ for $j \in d_{n+1}+2$ with the property that $\walkeq[\rho_n]{\bar{V}_i^{n,j}}{4 r_n}$ is $\rho_n$-uniformly bounded for all $i$ and $j$, and such that for every $x \in X$ and $i \in d_n+1$ there is at most one $j \in d_{n+1}+2$ with $\bar{V}_i^{n,j}$ dividing $B(x; 2 r_n)$.

Now for $j \in d_{n+1}+2$ define $\cV^{n,j} = \{V_0^{n,j}, \ldots, V_{d_n}^{n,j}\}$ where $V_i^{n,j}$ is the Borel set obtained by taking the union, over all $V \in \cV^{n-1,i}$, of the $\walkeq[\rho_{n-1}]{V}{2 r_{n-1}}$-classes meeting $\bar{V}_i^{n,j}$. Since $\mathcal{V}^{n-1,i}$ covers $X$ we have $V_i^{n,j} \supseteq \bar{V}_i^{n,j}$ and thus $\mathcal{V}^{n,j}$ covers $X$. Also, since $V_i^{n,j} \subseteq B_{\rho_n}(\bar{V}_i^{n,j}; r_n)$ it is immediate that $\walkeq[\rho_n]{V_i^{n,j}}{2 r_n}$ is $\rho_n$-uniformly bounded.

Define $\cU^n = \cV^{n, d_{n+1}+1}$. It only remains to check (ii). So fix $k \leq n$ and $x \in X$. If $k = n$ then this is immediate since $\cV^{n,j}$ divides $B_{\rho_n}(x; r_n)$ only when $\bar{\cV}^{n,j}$ divides $B_{\rho_n}(x; 2r_n)$. So assume $k < n$. Let $J$ be the set of $j \in d_{n+1}+2$ such that $\cV^{n,j}$ divides $B_{\rho_k}(x; r_k)$ and let $I$ be the set of $i \in d_n+1$ such that $\cV^{n-1,i}$ divides $B_{\rho_k}(x; r_k)$. By our inductive assumption it suffices to show that $|J| \leq |I|$.

Fix $j \in J$ and let $t(j) = i$ be any element of $d_n+1$ such that $V_i^{n,j}$ divides $B_{\rho_k}(x; r_k)$. By definition of $V_i^{n,j}$, there is some $V \in V^{n-1,i}$ with a class of $\walkeq[\rho_{n-1}]{V}{2 r_{n-1}}$ dividing $B_{\rho_k}(x; r_k)$. This is only possible if $V$ divides $B_{\rho_k}(x; r_k)$ as the latter set has $\rho_{n-1}$-diameter at most $2 r_{n-1}$. Thus $i \in I$. Additionally, $\bar{V}_i^{n,j}$ must divide $B_{\rho_n}(x; 2 r_n)$, and with $i$ and $x$ fixed there is at most one such value of $j \in d_{n+1} + 2$ with this property. Thus $t$ is an injection from $J$ to $I$.
\end{proof}

\begin{cor} \label{cor:free_weiss}
Suppose that $G$ is the increasing union of groups satisfying the assumption of Theorem \ref{thm:normal_asdim}. Then for every free Borel action $G \acts X$ on a standard Borel space $X$, the orbit equivalence relation $E_G^X$ is hyperfinite.
\end{cor}

The above corollary applies to all groups that locally have finite Pr{\"u}fer rank and are locally virtually solvable, all locally polycyclic groups, the Baumslag--Solitar group $BS(1, 2)$, the lamplighter group $\Z_2 \wr \Z$, and many groups not having finite standard asymptotic dimension, such as $(\bigoplus_{n \in \N} \Z) \rtimes S_\infty$ where $S_\infty$ is the group of finitely supported permutations of $\N$.

In \cite[Cor. 8.2]{SS} it was proven that if all free Borel actions of polycyclic groups generate hyperfinite Borel equivalence relations, then the same is true for all (not necessarily free) Borel actions of polycyclic groups. This result allows us to remove the restriction to free actions in the case of polycyclic groups. Similarly, Corollary \ref{cor:polygrow_asdim} allows us to remove this restriction in the case of groups having local polynomial volume-growth. We thus obtain a positive solution to Weiss' question in the case of these two classes of groups.

\begin{cor}
Let $G$ be either a polycyclic group or a countable group of local polynomial volume-growth. If $X$ is a standard Borel space and $G \acts X$ is any Borel action, then the orbit equivalence relation $E_G^X$ is hyperfinite.
\end{cor}

These two corollaries could be strengthened if Theorem \ref{thm:normal_asdim} were strengthened (as we discussed at the end of Section \ref{sec:series}). However, further progress on Weiss' question will eventually involve groups that do not have finite standard asymptotic dimension nor are unions of such groups. For instance, $\Z \wr \Z$ and the Grigorchuk group are not unions of groups having finite standard asymptotic dimension. Looking at the big picture, we feel that Borel asymptotic dimension is more than a mere property that implies hyperfiniteness, but rather the language surrounding this concept is meaningful and efficient at handling certain tasks related to hyperfiniteness proofs, as displayed in the proof of Theorem \ref{thm:union}. Thus, while the ability of Borel asymptotic dimension is strictly limited in providing further advances on Weiss' question, we see potential for the language surrounding this notion to evolve to handle more general situations.

\section{Borel graphs and chromatic numbers} \label{sec:chromatic}

The objects as in Lemma \ref{lem:equivalent_asym_dim}.(2') that witness $\asi(X, \rho) \leq 1$ are somewhat similar to objects called ``barriers'' in \cite{CM16} and ``toast'' in \cite{GJKSa,GJKSb,GJKSc}. In \cite{CM16,GJKSc} those objects were used to obtain bounds on Borel chromatic numbers. Below we generalize the proof by Conley--Miller in \cite{CM16} to our setting of asymptotic separation index. A peculiar aspect to our generalization is that we do not require $\asi(X, \rho)$ be at most $1$ but instead simply require it to be finite; this produces a dynamic bound rather than the customary bound of $2 \chi(\Gamma) - 1$, and further motivates the question as to whether $\asi(X, \rho)$ can be both finite and greater than one.

\begin{thm}
Let $X$ be a standard Borel space, let $\Gamma$ a Borel graph on $X$, and let $\rho$ be the graph metric on $\Gamma$. Assume that $\Gamma$ is locally finite and that $s = \asi(X, \rho) < \infty$. Then $\chi_\textbf{B}(\Gamma) \leq (s + 1) \cdot (\chi(\Gamma) - 1) + 1$.
\end{thm}

\begin{proof}
Let $U_0, \ldots, U_s$ be Borel sets covering $X$ with $\walkeq{U_i}{4}$ bounded for every $i$, and set $V_i = B(U_i; 1)$. Since $\walkeq{V_i}{2}$ is a finite Borel equivalence relation, by Proposition~\ref{lem:smooth_coloring} there is a Borel proper coloring $\pi_i \from V_i \to \chi(\Gamma)$ of $\Gamma \res V_i$.

Define $c \from X \to \{0\} \cup ((s+1) \times (\chi(\Gamma)-1))$ by setting $c(x) = 0$ if $\pi_i(x) = \chi(\Gamma)-1$ for every $i$ with $x \in V_i$, and otherwise set $c(x) = (i, \pi_i(x))$ where $i$ is least with $x \in V_i$ and $\pi_i(x) \neq \chi(\Gamma)-1$. Then $c$ is Borel. Now suppose $(x, y) \in \Gamma$. Then there is $i$ with $x \in U_i$ and thus $x, y \in V_i$, which implies $c(x)$ and $c(y)$ cannot both be $0$. Say $c(x) = (j, p) \neq 0$. Then either $y \not\in V_j$ or else $\pi_j(y) \neq p$, and in either case $c(y) \neq c(x)$.
\end{proof}

Combining the above theorem with Theorem \ref{thm:asi}.(a) produces the following.

\begin{cor} \label{cor:chrome}
Let $X$ be a standard Borel space, let $\Gamma$ a locally finite Borel graph on $X$, and let $\rho$ be the graph metric on $\Gamma$. If $\Basdim(X, \rho) < \infty$ then $\chi_\textbf{B}(\Gamma) \leq 2 \chi(\Gamma) - 1$.
\end{cor}

Notice that the above corollary provides bounds on the Borel chromatic numbers of Schreier graphs induced by free Borel actions of $G$ whenever $G$ is a finitely generated group satisfying the assumption of Theorem \ref{thm:normal_asdim}.

Lastly, we briefly discuss the Borel asymptotic dimension of Borel graphs in the highly restricted setting of acyclic graphs. It is not hard to see that when $\Gamma$ is acyclic $(X, \rho)$ has standard asymptotic dimension at most $1$. However, this can fail in the Borel case. For example, take the graph generated by the free part of the shift action of the free group $\mathbb{F}_2$ on $\N^{\mathbb{F}_2}$, and use \cite[Lemma 2.1]{M} to see that Lemma~\ref{lem:equivalent_asym_dim}.(2) fails for every $d$. Furthermore, in \cite{CJMST-D} hyperfinite acyclic bounded degree Borel graphs are constructed that have arbitrarily large Borel chromatic number and thus have infinite Borel asymptotic dimension by Corollary \ref{cor:chrome}. On the other hand, below we show that when a Borel graph is induced by a bounded-to-one Borel function it has Borel asymptotic dimension at most $1$. Recall that if $X$ is a standard Borel space and $f : X \rightarrow X$ is a Borel function, then $\Gamma_f$ is the Borel graph whose vertices are the elements of $X$ and where $x, y \in X$ are adjacent if $f(x) = y$ or $f(y) = x$. 

\begin{lem} \label{lem:bound_to_one}
Let $X$ be a standard Borel space, let $f : X \rightarrow X$ be a bounded-to-one Borel function, and let $\rho$ be the graph metric on $\Gamma_f$. Then $(X, \rho)$ has Borel asymptotic dimension at most $1$.
\end{lem}

\begin{proof}
Fix $r > 0$. We begin by finding a Borel set $Y \subseteq X$ which has nice recurrence properties for $f$ and $r$. Set $A_0 = X$ and $f_0 = f$. Inductively assume that a Borel set $A_n \subseteq X$ and a bounded-to-one Borel function $f_n : A_n \rightarrow A_n$ have been defined. Apply \cite[Lem. 4.1]{CJMST-D}\footnote{The published proof of \cite[Lem. 4.1]{CJMST-D} incorrectly assumes $\forall x \in X \ f(x) \neq x$, but the result is true as stated. To correct the proof, it suffices to make two changes. First, require the function $c$ referred to in that proof have the additional property that for every $x \in X$, if $f(x) = x$ then $c(x) = 0$, and if there is $n \geq 1$ with $f^{n+1}(x) = f^n(x) \neq f^{n-1}(x)$ then $c(x) \in \{0,1\}$ must be congruent to $n$ mod $2$. Second, include in the set $A$ all points $x$ satisfying $f(x) = x$.} to find a Borel set $A_{n+1} \subseteq A_n$ such that $A_{n+1}$ shares at most one point with $\{x, f_n(x)\}$ and at least one point with $\{f_n(x), f_n^2(x), f_n^3(x), f_n^4(x)\}$ for every $x \in A_n$, and for $x \in A_{n+1}$ set $f_{n+1}(x) = f_n^i(x)$ where $i > 0$ is least with $f_n^i(x) \in A_{n+1}$. Pick $k \in \N$ with $2^k > 2 r$ and set $Y = A_k$. Then for some $m \in \N$ (one can use $m = 4^k$) we have that for every $x \in X$,
$$|Y \cap \{x, f(x), \ldots, f^{2r}(x)\}| \leq 1 \leq |Y \cap \{f^i(x) : 1 \leq i \leq m\}|.$$

Let $g \from X \to Y$ be defined by $g(x) = f^i(x)$ where $i > 0$ is least such that $f^i(x) \in Y$. Let $E$ be the equivalence relation on $X$ where $x \mathrel{E} y$ if and only if $g(f^r(x)) = g(f^r(y))$. Then $E$ is uniformly bounded since $\rho(x, g(x)) \leq m$ for all $x$. Lastly, we claim that for every $x \in X$ the ball $B(x; r)$ meets at most $2$ classes of $E$. Indeed, $f^r(B(x;r)) \subseteq \{x, f(x), \ldots, f^{2r}(x)\}$ and $g$ can take at most $2$ values on $\{x, f(x), \ldots, f^{2r}(x)\}$ since $|Y \cap \{f(x), \ldots, f^{2r+1}(x)\}| \leq 1$.
\end{proof}

Lemma~\ref{lem:bound_to_one} cannot be generalized to finite-to-one Borel functions. For example, let $[\N]^{\infty}$ be the standard Borel space of infinite subsets of $\N$, and let $f \from [\N]^{\infty} \to [\N]^{\infty}$ be the shift function where $f(A) = A \setminus \{\min(A)\}$. Letting $\rho$ be the graph metric on $\Gamma_f$, we have that $([\N]^{\infty}, \rho)$ does not have finite Borel asymptotic dimension. This follows from condition (2) of Lemma~\ref{lem:equivalent_asym_dim} and the Galvin-Prikry theorem \cite[Thm. 19.11]{K95}. Specifically, if $U_0, \ldots, U_d$ are Borel sets covering $[\N]^\infty$, then there is some $A \in [\N]^\infty$ and $i \in d+1$ with $f^n(A) \in U_i$ for all $n$.

\section{F{\o}lner tilings} \label{sec:folner}

For an action $G \acts X$, a nonempty finite set $K \subseteq G$, and $\delta > 0$, we say that a finite set $F \subseteq X$ is \emph{$(K, \delta)$-invariant} if $|(K \cdot F) \symd F| < \delta |F|$. We similarly say that a finite set $F \subseteq G$ is $(K, \delta)$-invariant if $|(K F) \symd F| < \delta |F|$. Recall from the F{\o}lner characterization of amenability that a countable group $G$ is \emph{amenable} if and only if for every finite nonempty set $K \subseteq G$ and $\delta > 0$ there exists a finite $(K, \delta)$-invariant subset of $G$. The concept of amenability appears as a common theme throughout much of ergodic theory, underlying the classical ergodic theorems, Kolmogorov--Sinai entropy theory, and the foundation of orbit equivalence theory.

A key method for applying amenability in dynamics is the use of F{\o}lner tilings (i.e. partitions of the space into finite sets that are $(K, \delta)$-invariant for some prescribed $K$ and $\delta$). In ergodic theory, the weaker property of F{\o}lner quasi-tilings often suffices and was established in the seminal work of Ornstein and Weiss that extended Kolmogorov--Sinai entropy theory and Ornstein theory to all countable amenable groups \cite{OW87}. Outside of ergodic theory, quasi-tilings no longer suffice and stronger types of tilings are needed. This issue is receiving renewed attention after the recent proof by Downarowicz, Huczek, and Zhang of the purely group-theoretic fact that every countable amenable group can be tiled by F{\o}lner sets\cite{DHZ19}. This led to the discovery that free measure-preserving actions of countable amenable groups always admit measurable F{\o}lner tilings on an invariant conull set \cite{CJKMST-D18}, and that continuous actions of groups of subexponential volume-growth on $0$-dimensional compact metric spaces always admit clopen F{\o}lner tilings \cite{DZ}.

In this brief section we show that finite Borel asymptotic dimension (or more generally, finite asymptotic separation index) is sufficient for the existence of Borel F{\o}lner tilings. A similar topological result providing clopen F{\o}lner tilings is proven in the next section. Combined with Theorem \ref{thm:normal_asdim} this produces a large class of new groups whose actions admit F{\o}lner tilings.

\begin{thm} \label{thm:folner}
Let $G$ be a countable amenable group, let $X$ be a standard Borel space, and let $G \acts X$ be a free Borel action. Assume that $\asi(H \acts X) < \infty$ for every finitely generated subgroup $H \leq G$. Then for every finite nonempty $K \subseteq G$ and $\delta > 0$ there exist $(K, \delta)$-invariant finite sets $F_1, \ldots, F_n \subseteq G$ and Borel sets $C_1, \ldots, C_n \subseteq X$ such that the map $\theta : \bigsqcup_{i=1}^n F_i \times C_i \rightarrow X$ given by $\theta(f, c) = f \cdot c$ is a bijection.
\end{thm}

\begin{proof}
Without loss of generality, we may assume that $1_G \in K$ and thus a set $F$ is $(K, \delta)$-invariant if and only if $|(K \cdot F) \setminus F| < \delta |F|$. Set $H = \langle K \rangle$, $d = \asi(H \acts X)$, and $K_0 = K$.

Inductively assume that finite sets $K_0, \ldots, K_i \subseteq H$ have been defined, where $i \in d+1$. By \cite{DHZ19} there exists a tiling of $H$ by $(K_0 \cdots K_i, \delta)$-invariant sets. More specifically, there are finite $(K_0 \cdots K_i, \delta)$-invariant sets $A_0^i, \ldots, A_{\ell_i}^i \subseteq H$, all containing $1_G$, and sets $T_0^i, \ldots, T_{\ell_i}^i \subseteq H$ such that $\{A_p^i t : p \leq \ell_i, \ t \in T_p^i\}$ partitions $H$. Now set $K_{i+1} = K_i \cup \bigcup_{p \leq \ell_i} A_p^i (A_p^i)^{-1} \subseteq H$ and continue the induction if $i < d$.

Next, set $B = K_{d+1}^{-1} K_{d+1} \subseteq H$ and invoke the definition of $d$ to obtain a Borel cover $\{U_0, \ldots, U_d\}$ of $X$ such that $\walkeq{U_i}{B}$ is finite for every $i$. Let $M_i \subseteq U_i$ be a Borel set containing exactly one point from every $\walkeq{U_i}{B}$-class. For each point $x \in M_i$ we can use the tiling of $H$ by the shapes $A_0^i, \ldots, A_{\ell_i}^i$ to obtain a tiling of $H \cdot x$. By restricting to those tiles that intersect the $\walkeq{U_i}{B}$-equivalence class of $x$ and letting $x \in M_i$ vary, we can obtain a tiling of a set $U_i^+$ containing $U_i$. Specifically, define the Borel sets
$$Z_p^i = \{t \cdot x : x \in M_i, \ t \in T_p^i, \text{ and } A_p^i t \cdot x \cap [x]_{\walkeq{U_i}{B}} \neq \varnothing\}$$
and define $U_i^+ = \bigcup_{p \leq \ell_i} A_p^i \cdot Z_p^i$. Notice that $U_i \subseteq U_i^+ \subseteq K_{d+1} \cdot U_i$ and that the sets $\{A_p^i \cdot z : p \leq \ell_i, \ z \in Z_p^i\}$ partition $U_i^+$.

We now have a tiling of each set $U_i^+$, but a problem we need to overcome is that these sets might not be pairwise disjoint. We first shrink the $U_i^+$'s to pairwise disjoint sets $V_i$ by removing tiles from $U_i^+$ that intersect some $U_j^+$ for $j > i$. Specifically, let $Y_p^i$ consist of those points $z \in Z_p^i$ for which $A_p^i \cdot z$ is disjoint with $\bigcup_{j > i} U_j^+$ and set $V_i = \bigcup_{p \leq \ell_i} A_p^i \cdot Y_p^i$. Then the $V_i$'s are pairwise disjoint and each $V_i$ is partitioned by $\{A_p^i \cdot y : p \leq \ell_i, \ y \in Y_p^i\}$.

The last thing we need to do is to assign each point in $X \setminus \bigcup_{i \in d+1} V_i$ to some tile of some $V_i$ in a way that doesn't interfere with the almost-invariance properties of the tiles. To do this, we will define a Borel function $s : X \rightarrow \bigcup_{i \in d+1} V_i$ that is the identity on $\bigcup_{i \in d+1} V_i$. To begin, for $x \in U_d^+ = V_d$ set $s(x) = x$. Now inductively assume that $s$ is defined on $\bigcup_{j > i} U_j^+$ and consider $x \in U_i^+$. If $x \in V_i$ then set $s(x) = x$. Otherwise pick $p \leq \ell_i$ and $z \in Z_p^i$ with $x \in A_p^i \cdot z$ and, using any fixed enumeration of $G$, set $s(x) = s(a \cdot z)$ where $a$ is the least element of $A_p^i$ satisfying $a \cdot z \in \bigcup_{j > i} U_j^+$. Notice that $s$ is Borel, $s(x) = x$ for all $x \in \bigcup_{i \in d+1} V_i$, and when $x \in U_i^+$ and $x \neq s(x) \in V_j$ we have $j > i$ and $x \in K_{i+1} \cdots K_j \cdot s(x)$.

Using the assignment $s$ and the tilings of the $V_i$'s, we obtain a Borel equivalence relation $E$ by declaring $(x, x') \in E$ if and only if there is $i$ with $s(x), s(x') \in V_i$ and there are $p \leq \ell_i$ and $y \in Y_p^i$ with $s(x), s(x') \in A_p^i \cdot y$. In other words, the classes of $E$ consist of the sets $s^{-1}(A_p^i \cdot y)$ for $i \in d+1$, $p \leq \ell_i$, and $y \in Y_p^i$. Since $A_p^i \cdot y \subseteq s^{-1}(A_p^i \cdot y) \subseteq K_1 \cdots K_i A_p^i \cdot y$, we have
$$|(K \cdot s^{-1}(A_p^i \cdot y)) \setminus s^{-1}(A_p^i \cdot y)| \leq |(K_0 K_1 \cdots K_i A_p^i) \setminus A_p^i| < \delta |A_p^i| \leq \delta |s^{-1}(A_p^i \cdot y)|,$$
and thus every $E$-class is $(K, \delta)$-invariant.

Notice that $[y]_E \subseteq K_1 \cdots K_{d+1} \cdot y$ for every $y \in \bigcup_{i \in d+1} \bigcup_{p \leq \ell_i} Y_p^i$. Let $F_1, F_2, \ldots, F_n$ enumerate the $(K, \delta)$-invariant subsets of $K_1 \cdots K_{d+1}$. For every $m \leq n$ let $C_m$ be the set of $y \in \bigcup_{i \in d+1} \bigcup_{p \leq \ell_i} Y_p^i$ satisfying $[y]_E = F_m \cdot y$. Then $C_m$ is Borel since $y \in C_m$ if and only if $y \in \bigcup_{i \in d+1} \bigcup_{p \leq \ell_i} Y_p^i$, $(y, f \cdot y) \in E$ for every $f \in F_m$, and $(y, h \cdot y) \not\in E$ for every $h \in (K_1 \cdots K_{d+1}) \setminus F_m$. Finally, the map $\theta : \bigsqcup_{i=1}^n F_i \times C_i \rightarrow X$ given by $\theta(f, c) = f \cdot c$ is a bijection since the sets $C_1, \ldots, C_n$ are pairwise disjoint, $\bigcup_{i=1}^n C_i$ contains precisely one point from every $E$-class, and $\theta$ bijects $F_i \times \{c\}$ with $[c]_E$ for every $1 \leq i \leq n$ and $c \in C_i$.
\end{proof}

\section{Topological dynamics and \texorpdfstring{$C^*$}{C*}-algebras} \label{sec:calg}

In this last section we observe that the proof of our main theorem, Theorem \ref{thm:normal_asdim}, easily adapts to the realm of topological dynamics. Moreover, we will see that the proof of Theorem \ref{thm:folner} adapts as well and that together these results have important consequences to $C^*$-algebras.

We first record a basic lemma.

\begin{lem} \label{lem:clopen}
Let $X$ be a $0$-dimensional second countable Hausdorff space and let $G \acts X$ be a free continuous action. Let $U \subseteq X$ be clopen and let $B \subseteq G$ be finite and suppose that $\walkeq{U}{B}$ is uniformly finite. Then:
\begin{enumerate}
\item[(i)] for every $g \in G$ the set $\{x \in U \cap g^{-1} \cdot U : (x, g \cdot x) \in \walkeq{U}{B}\}$ is clopen; and
\item[(ii)] there is a clopen set $Y \subseteq U$ that contains exactly one point from every $\walkeq{U}{B}$-class.
\end{enumerate}
\end{lem}

\begin{proof}
Since $\walkeq{U}{B}$ is uniformly finite, there is a finite set $A \subseteq G$ with $[x]_{\walkeq{U}{B}} \subseteq A \cdot x$ for all $x \in X$ (specifically, one can take $A = (B \cup B^{-1} \cup \{1_G\})^n$ for sufficiently large $n$).

(i). Let $P$ be the set of all finite sequences $a_0, a_1, \ldots, a_n \in A$ satisfying $a_0 = 1_G$, $a_i \neq a_j$ for all $i \neq j$, and satisfying either $a_{i+1} \in B a_i$ or $a_i \in B a_{i+1}$ for every $i < n$. Then $P$ is finite and for $g \in G$ and $x \in U \cap g^{-1} \cdot U$ we have $(x, g \cdot x) \in \walkeq{U}{B}$ if and only if there is $(a_0, \ldots, a_n) \in P$ with $a_n = g$ and satisfying $x \in a_i^{-1} \cdot U$ for every $0 \leq i \leq n$. Thus $\{x \in U \cap g^{-1} \cdot U : (x, g \cdot x) \in \walkeq{U}{B}\}$ is clopen.

(ii). Since the action is free, for any $x \in X$ we have that $|A \cdot x| = |A|$ and therefore for every open set $V$ containing $x$ there is a clopen set $V' \subseteq V$ containing $x$ with the sets $a \cdot V'$, $a \in A$, pairwise disjoint. So there is a base $\mathcal{V}$ for the topology on $X$ such that each $V \in \mathcal{V}$ is clopen and the sets $a \cdot V$, $a \in A$, are pairwise disjoint. Since $X$ is second countable, there is a countable base $\{V_n : n \in \N\} \subseteq \mathcal{V}$ for the topology on $X$. Notice that no two points of $V_n \cap U$ can be $\walkeq{U}{B}$-equivalent. Define $Y_n$ to be the set of $v \in V_n \cap U$ with $[v]_{\walkeq{U}{B}} \cap \bigcup_{i < n} V_i = \varnothing$. Then $Y_n$ is clopen since $v \in Y_n$ if and only if $v \in V_n \cap U$ and for every $a \in A$ we either have $a \cdot v \not\in U$, $(v, a \cdot v) \notin \walkeq{U}{B}$, or $a \cdot v \not\in \bigcup_{i < n} V_i$. Therefore $Y = \bigcup_n Y_n$ is open. Moreover, $Y$ is also closed since $X \setminus Y$ is the union of the clopen sets $V_n \setminus (Y_n \cup \bigcup_{i < n} V_i)$. Lastly, it is immediate from the definition that $Y$ meets every $\walkeq{U}{B}$-class precisely once.
\end{proof}

In the realm of topological dynamics, the appropriate analog of asymptotic dimension is the notion of dynamic asymptotic dimension as defined by Guentner, Willet, and Yu in \cite{GWY17}. Specifically, for a discrete group $G$ acting continuously on a locally compact Hausdorff space $X$, the \emph{dynamic asymptotic dimension} is $d$ if $d \in \N$ is least with the property that for any compact set $K \subseteq X$ and finite set $B \subseteq G$ there are open sets $U_0, \ldots, U_d$ that cover $K$ such that for every $i$ the set $T^{U_i}_B$ is finite, where $g \in T^{U_i}_B$ if and only if there are $x \in U_i$, $m \geq 1$, and $b_1, \ldots, b_m \in B \cup B^{-1}$ with $g = b_m b_{m-1} \cdots b_1$ and $b_k b_{k-1} \cdots b_1 \cdot x \in U_i$ for every $1 \leq k \leq m$. When no such $d$ exists the dynamic asymptotic dimension is defined to be $\infty$. When the action is free (as we will always assume), one can instead simply require that $\walkeq{U_i}{B}$ be uniformly finite for every $i$.

For our purposes, we will consider free continuous actions on $0$-dimensional spaces and will want the stronger property that the sets $U_0, \ldots, U_d$ are clopen and cover all of $X$ rather than a given compact subset of $X$. We call the least such $d$ with this property the \emph{strong dynamic asymptotic dimension} of the action (this terminology is introduced solely for purposes of convenience for the remainder of this section). It is immediate from the definitions that for a free continuous action on a $0$-dimensional locally compact Hausdorff space the strong dynamic asymptotic dimension is greater than or equal to the dynamic asymptotic dimension. We first make the easy observation that these concepts coincide when $X$ is additionally compact.

\begin{lem} \label{lem:compact}
Let $G$ be a countable group, $X$ a $0$-dimensional compact Hausdorff space, and $G \acts X$ a free continuous action. Then the strong dynamic asymptotic dimension and the dynamic asymptotic dimension of this action are equal.
\end{lem}

\begin{proof}
Let $d$ be the dynamic asymptotic dimension of the action. The strong dynamic asymptotic dimension must be at least $d$. So we assume $d < \infty$ and will show that the strong dynamic asymptotic dimension is at most $d$. Let $B \subseteq G$ be finite. Since $X$ is compact and the dynamic asymptotic dimension is $d$, there is an open cover $\{V_0, \ldots, V_d\}$ of $X$ such that $\walkeq{V_i}{B}$ is uniformly finite for every $i$. It is easily seen that any disjoint pair of closed subsets of $X$ is separated by a clopen set. So we can define clopen sets $U_i$ inductively by requiring $U_i \subseteq V_i$ to be a clopen set containing $X \setminus (\bigcup_{j < i} U_j \cup \bigcup_{j > i} V_j)$. Then the $U_i$'s are clopen and cover $X$, and $\walkeq{U_i}{B}$ is uniformly finite since $U_i \subseteq V_i$. Thus the strong dynamic asymptotic dimension is at most $d$.
\end{proof}

The following lemma is a topological analog of Theorem \ref{thm:equal}. We remark that in this setting we only obtain a bound on dynamic asymptotic dimension rather than equality. This is due to the fact that for actions on non-compact spaces, dynamic asymptotic dimension can be less than the standard asymptotic dimension of the acting group \cite[Rem. 2.3.(ii)]{GWY17}.

\begin{lem} \label{lem:clopen_equal}
Let $G$ be a countable group, $X$ a $0$-dimensional second countable Hausdorff space, and let $G \acts X$ be a free continuous action. Then the strong dynamic asymptotic dimension of this action is either infinite or equal to $\asdim(G)$. In particular, when the strong dynamic asymptotic dimension is finite and $X$ is in addition locally compact, the dynamic asymptotic dimension of the action $G \acts X$ is at most the standard asymptotic dimension of $G$.
\end{lem}

\begin{proof}
It is immediate from the definitions that the strong dynamic asymptotic dimension must be greater than or equal to $\asdim(G)$ (since the action is free). So it suffices to assume that the strong dynamic asymptotic dimension is finite and show that it is at most $\asdim(G)$. To do so, we will retrace the proof of Theorem \ref{thm:equal} and the supporting Lemmas \ref{lem:smooth_asdim} and \ref{lem:finiteunion} with clopen sets in place of Borel sets. Fix any proper right-invariant metric $\tau$ on $G$, and let $\rho$ be the extended metric on $X$ satisfying $\rho(g \cdot x, h \cdot x) = \tau(g, h)$ for all $g, h \in G$ and $x \in X$ and satisfying $\rho(x, y) = \infty$ whenever $G \cdot x \neq G \cdot y$. Notice that by definition $\asdim(G)$ is equal to $\asdim(G, \tau) = \asdim(X, \rho)$.

Let's say that a clopen set $Y \subseteq X$ has clopen $(r, R)$-dimension $q$ if there exist clopen sets $U_0, \ldots, U_q$ covering $Y$ such that for every $i$ each class of $\walkeq{U_i}{r}$ has diameter at most $R$. In analogy with Lemma \ref{lem:smooth_asdim}, we claim that if $Y \subseteq X$ is clopen, $\walkeq{Y}{r}$ is uniformly bounded, and $(G, \tau)$ has $(r, R)$-dimension $d$, then $Y$ has clopen $(r, R)$ dimension at most $d$. Indeed, by Lemma \ref{lem:clopen} there is a clopen set $Y_* \subseteq Y$ that meets every $\walkeq{Y}{r}$-class in exactly one point. Let $\{T_0, \ldots, T_d\}$ be any partition of $G$ with the property that for every $i$ each class of $\walkeq{T_i}{r}$ has diameter at most $R$. Let $w > 0$ bound the diameter of every $\walkeq{Y}{r}$-class, define the finite sets $T_i' = \{t \in T_i : \tau(1_G, t) \leq w\}$, and set
$$U_i = \bigcup_{t \in T_i'} \{t \cdot y : y \in Y_* \cap t^{-1} \cdot Y \text{ and } (y, t \cdot y) \in \walkeq{Y}{r}\}.$$
Then $U_0, \ldots, U_d$ are clopen by Lemma \ref{lem:clopen} and they partition $Y$. Additionally, any $\walkeq{U_i}{r}$-class is a subset of a $\walkeq{T_i \cdot y}{r}$-class for some $y \in Y_*$ and thus has diameter at most $R$. So $Y$ has clopen $(r, R)$-dimension at most $d$ as claimed.

Next, in analogy with Lemma \ref{lem:finiteunion}, we claim that if $X$ is partitioned by the clopen sets $X_0, \ldots, X_{n-1}$ and each $X_i$ has clopen $(15r_i, r_{i+1})$-dimension at most $d$, where the $r_n$'s are non-decreasing, then $(X, \rho)$ has clopen $(r_0, 5 r_n)$-dimension at most $d$. Indeed, one can follow the proof of Lemma \ref{lem:finiteunion} and change each instance of ``Borel'' to ``clopen.''

Finally, we retrace the proof of Theorem \ref{thm:equal}. Let $D$ be the strong dynamic asymptotic dimension of $G \acts X$, set $d = \asdim(G)$, and let $B \subseteq G$ be finite. Since $\asdim(X, \rho) = d$ we can find an increasing sequence of radii $r_i$ such that $r_0 > \max \{\tau(b, 1_G) : b \in B\}$ and such that $X$ has standard $(15 r_i, r_{i+1})$-dimension at most $d$ for every $i$. By definition of $D$, we can find a clopen partition $\{X_0, \ldots, X_D\}$ of $X$ such that $\walkeq{X_i}{r_D}$ is uniformly bounded for every $i$. Our analog of Lemma \ref{lem:smooth_asdim} implies that each $X_i$ has clopen $(15 r_i, r_{i+1})$-dimension at most $d$, and thus our analog of Lemma \ref{lem:finiteunion} implies that $(X, \rho)$ has clopen $(r_0, 5 r_{D+1})$-dimension at most $d$. This means that there exists a clopen cover $\{U_0, \ldots, U_d\}$ of $X$ such that for every $i$ each class of $\walkeq{U_i}{r_0}$ has diameter at most $5 r_{D+1}$. From our definition of $r_0$ we then see that $\walkeq{U_i}{B}$ is uniformly finite for every $i$. This shows that the strong dynamic asymptotic dimension of the action is at most $d$ as claimed.
\end{proof}

We next turn to proving a topological version of Theorem \ref{thm:normal_asdim}. We will proceed much in the same way as before.

\begin{defn} \label{defn:good_clopen}
For countable groups $G \lhd H$, let's say that $(G, H)$ has \emph{good clopen asymptotics} if for every $0$-dimensional second countable Hausdorff space $X$ and every free continuous action $H \acts X$ the following two properties hold:
\begin{enumerate}
\item[(i)] the restricted action $G \acts X$ has finite strong dynamic asymptotic dimension; and
\item[(ii)] for every finite set $C \subseteq H$ there is $\ell \in \N$ so that for every finite set $B \subseteq G$ there is a clopen cover $\{W_0, \ldots, W_\ell\}$ of $X$ consisting of $C B \setminus G$-independent sets.
\end{enumerate}
\end{defn}

\begin{lem} \label{lem:clopen_expand}
Let $F \lhd H$ be countable groups with $(F, H)$ having good clopen asymptotics. Let $X$ be a $0$-dimensional second countable Hausdorff space, let $H \acts X$ be a free continuous action, and let $d < \infty$ be the strong dynamic asymptotic dimension of the restricted action $G \acts X$. Then for any finite sets $C \subseteq H$ and $A \subseteq F$ there is a clopen cover $\{U_i^n : i \in d+1, n \in \ell+1\}$ of $X$ such that $U_i^n$ is $C A \setminus F$-independent and $\walkeq{U_i^n}{A}$ is uniformly finite, and there are finite sets $C A \cap F \subseteq A_n \subseteq F$ such that whenever $Z$ is a $\walkeq{U_i^n}{A_n}$-class, $c \in C$, and $m > n$, the set $c A \cdot Z$ is either disjoint with $U_i^m$ or else contained in a single $\walkeq{U_i^m}{A_m}$-class.
\end{lem}

\begin{proof}
We argue that the proof of Lemma \ref{lem:expand} can be followed but with clopen sets in place of Borel sets. Specifically, in that proof we can choose the sets $V_i^n$ and $W_n$ to be clopen rather than Borel. Notice that the explanation given there for why $s_i^n(X')$ is Borel when $X'$ is Borel shows, when combined with Lemma \ref{lem:clopen}, that $s_i^n(X')$ is clopen when $X'$ is clopen. Therefore the sets $U_i^0 = V_i^0 \cap W_0$ and $U_i^n = s_i^0 \circ \cdots \circ s_i^{n-1}(V_i^n \cap W_n)$, $n \geq 1$, are clopen.
\end{proof}

\begin{lem} \label{lem:clopen_polystep}
Let $H$ be a countable group and let $F \lhd G$ be normal subgroups of $H$. Assume that $(F, H)$ has good clopen asymptotics and that $G / F$ has $\{\phi_c : c \in C\}$-bounded packing whenever $C \subseteq H$ is a finite set with the property that each automorphism $\phi_c(g F) = c g c^{-1} F$ is either trivial or an outer automorphism. Then $(G, H)$ has good clopen asymptotics.
\end{lem}

\begin{proof}
Let $X$ be a $0$-dimensional second countable Hausdorff space and let $H \acts X$ be a free continuous action. Let $d_F$ be the strong dynamic asymptotic dimension of the restricted action $F \acts X$. We follow the proof of Lemma \ref{lem:polystep}. By replacing Lemma \ref{lem:expand} with Lemma \ref{lem:clopen_expand}, we can choose the sets $U_i^n$ to be clopen. It is immediate that the sets $V_i^{\ell_F} = U_i^{\ell_F}$ and $V_i^n = U_i^n \setminus \bigcup_{m > n} B \cdot V_i^m$ are clopen. Since the value $\pi_i(x)$, for $x \in V_i^n$, is defined to be $0$ if $n = \ell_F$ and for $n < \ell_F$ is defined to be the least element of $(k_{G/F}+1) \setminus \{\pi_i(t \cdot x) : t \in B^3 C B^3 \text{ and } t \cdot x \in \bigcup_{m > n} V_i^m\}$, it follows immediately from induction that the function $\pi_i$ is continuous. Therefore the sets $W_i^j = B \cdot \pi_i^{-1}(j)$ are clopen. The concluding arguments of the proof of Lemma \ref{lem:polystep} then show that $(G, H)$ has good clopen asymptotics.
\end{proof}

\begin{thm} \label{thm:normal_clopen}
Let $G$ be a countable group admitting a normal series where each quotient of consecutive terms is a finite group, an increasing union of finite characteristic subgroups, or a torsion-free abelian group with finite $\Q$-rank, except the top quotient can be any group of uniform local polynomial volume-growth or the lamplighter group $\Z_2 \wr \Z$. If $X$ is a $0$-dimensional second countable Hausdorff space and $G \acts X$ is a continuous free action then the strong dynamic asymptotic dimension of $G \acts X$ is equal to $\asdim(G) < \infty$. In particular, all free continuous actions of $G$ on locally compact $0$-dimensional second countable Hausdorff spaces have dynamic asymptotic dimension at most $\asdim(G) < \infty$.
\end{thm}

\begin{proof}
Follow the proof of Theorem \ref{thm:normal_asdim}, but with Lemma \ref{lem:clopen_polystep} replacing Lemma \ref{lem:polystep} and Lemma \ref{lem:clopen_equal} replacing Theorem \ref{thm:equal}. For verifying the base case of the induction, recall from the proof of Lemma \ref{lem:clopen} that, given a finite set $C \subseteq G$, there is a base $\{V_n : n \in \N\}$ for the topology on $X$ such that each $V_n$ is clopen and $c \cdot V_n \cap V_n = \varnothing$ for all $c \in C \setminus \{1_G\}$. Inductively define a continuous function $\pi : X \rightarrow |C|+1$ by letting $\pi(x)$ be the least element of $(|C| + 1) \setminus \{\pi(c \cdot x) : c \in C, \ c \cdot x \in \bigcup_{k < n} V_k\}$ for $x \in V_n \setminus \bigcup_{k < n} V_k$. Then $\{\pi^{-1}(i) : i \in |C|+1\}$ is a clopen cover of $X$ consisting of $C \setminus \{1_G\}$-independent sets.
\end{proof}

The theorem below together with Lemma \ref{lem:compact} and Theorem \ref{thm:folner} completes the proof of Theorem \ref{intro:folner} from the introduction.

\begin{thm} \label{thm:clopen_folner}
Let $G$ be a countable amenable group, let $X$ be a $0$-dimensional second countable Hausdorff space, and let $G \acts X$ be a free continuous action.  Assume that for every finitely generated subgroup $H \leq G$ the restricted action $H \acts X$ has finite strong dynamic asymptotic dimension. Then for every finite $K \subseteq G$ and $\delta > 0$ there exist $(K, \delta)$-invariant finite sets $F_1, \ldots, F_n \subseteq G$ and clopen sets $C_1, \ldots, C_n \subseteq X$ such that the map $\theta : \bigsqcup_{i=1}^n F_i \times C_i \rightarrow X$ given by $\theta(f, c) = f \cdot c$ is a bijection.
\end{thm}

\begin{proof}
We follow the proof of Theorem \ref{thm:folner} but use clopen sets. As before let $H = \langle K \rangle$ and let $d$ be the strong dynamic asymptotic dimension of $H \acts X$. We can then choose the sets $U_0, \ldots, U_d$ to be clopen and with the property that $\walkeq{U_i}{B}$ is uniformly finite for every $i$. Consequently, there is a finite set $Q \subseteq G$ with $K_{d+1} \cdot [x]_{\walkeq{U_i}{B}} \subseteq Q \cdot x$ for every $x \in U_i$ (one can use $Q = (B \cup B^{-1} \cup \{1_G\})^n$ for sufficiently large $n$). Next we use Lemma \ref{lem:clopen} to choose the $M_i$'s to be clopen. Since $T_p^i$ can be replaced with the finite set $T_p^i \cap Q$ in the definition of $Z_p^i$, we see that each $Z_p^i$ is clopen by Lemma \ref{lem:clopen}. Similarly, $U_i^+$, $Y_p^i$, and $V_i$ are clopen.

Next, we claim that the function $t : X \rightarrow H$ defined by the condition $s(x) = t(x) \cdot x$ is continuous. For $x \in U_d^+ = V_d$ we have $s(x) = x$ and thus $t(x) = 1_G$. So $t$ is continuous on $U_d^+$. Now inductively assume that $t$ is continuous on $\bigcup_{j > i} U_j^+$. For $x \in V_i$ we have $s(x) = x$ and $t(x) = 1_G$. When $x \in U_i^+ \setminus V_i$ we pick $p \leq \ell_i$ and $a_1 \in A_p^i$ satisfying $a_1^{-1} \cdot x \in Z_p^i$, we pick the least $a_2 \in A_p^i$ satisfying $a_2 a_1^{-1} \cdot x \in \bigcup_{j > i} U_j^+$, and set $s(x) = s(a_2 a_1^{-1} \cdot x)$, meaning $t(x) = t(a_2 a_1^{-1} \cdot x) a_2 a_1^{-1}$. Since $i$, $p$, $a_1$, and $a_2$ depend continuously on $x$, we conclude that $t$ is continuous.

It only remains to check that each set $C_m$ constructed as in the proof of Theorem \ref{thm:folner} is clopen. This is the case because $y \in C_m$ if and only if there is $i \in d+1$ and $p \leq \ell_i$ with $y \in Y_p^i$ and with the property that $t(f \cdot y) \in A_p^i f^{-1}$ for every $f \in F_m$ and $t(h \cdot y) \not\in A_p^i h^{-1}$ for all $h \in (K_1 \cdots K_{d+1}) \setminus F_m$.
\end{proof}

The combination of Theorems \ref{thm:normal_clopen} and \ref{thm:clopen_folner} allow us to find new examples of actions that are almost finite. From this property we derive consequences to $C^*$-algebras, ultimately finding new examples of classifiable crossed products.

Recall that for an action $G \acts X$, a set $A \subseteq X$ is topologically small if there exists $n \in \N$ such that whenever $g_1, \ldots, g_n \in G$ are distinct we have $\bigcap_{i=1}^n g_i \cdot A = \varnothing$. If $X$ is a topological space and $G$ acts by homeomorphisms, the action is said to have the \emph{topological small boundary property} if there is a basis for the topology on $X$ consisting of open sets whose boundaries are topologically small. We remark that every free continuous action of a countable group on a compact metrizable space with finite covering dimension automatically has the topological small boundary property \cite[Thm. 3.8]{Szabo}.

\begin{cor} \label{cor:cstar1}
Let $G$ be a countably infinite group that is an increasing union of groups satisfying the assumption of Theorem \ref{thm:normal_clopen}. Let $X$ be a compact metrizable space and let $G \acts X$ be a free continuous action. Assume either that $X$ has finite covering dimension or, more generally, that the action $G \acts X$ has the topological small boundary property. Then:
\begin{enumerate}
\item[(i)] the action $G \acts X$ is almost finite;
\item[(ii)] if additionally $G \acts X$ is minimal, then $C(X) \rtimes G$ is $\mathcal{Z}$-stable and has finite nuclear dimension.
\end{enumerate}
\end{cor}

\begin{proof}
Statement (i) follows from Theorem \ref{thm:normal_clopen}, Theorem \ref{thm:clopen_folner}, and \cite[Thm. 10.2]{Kerr} when $X$ is $0$-dimensional, and from \cite[Thm. 7.6 and Cor. 7.7]{KeSz} in the remaining cases. When the action is minimal, $\mathcal{Z}$-stability follows from (i) and \cite[Thm. 12.4]{Kerr}, and finite nuclear dimension follows from $\mathcal{Z}$-stability and \cite[Thm. A]{CETWW}.
\end{proof}

\begin{cor}
Consider the collection of free minimal continuous actions $G \acts X$ where $G$ is a countably infinite group that is an increasing union of groups satisfying the assumption of Theorem \ref{thm:normal_clopen}, $X$ is a compact metrizable space, and either $X$ has finite covering dimension or the action $G \acts X$ has the topological small boundary property. Then the crossed products arising from these actions are classified by the Elliot invariant (ordered $K$-theory paired with tracial states) and are simple ASH algebras of topological dimension at most $2$.
\end{cor}

\begin{proof}
The crossed products $C(X) \rtimes G$ of these actions have finite nuclear dimension by Corollary \ref{cor:cstar1}, satisfy the universal coefficient theorem (since the groups we consider are amenable) \cite[Prop. 10.7]{Tu99}, and are stably finite (since they come from free minimal actions of amenable groups). Therefore these crossed products are classified by their Elliot invariants \cite[Cor. D]{TWW} and are simple ASH algebras of topological dimension at most $2$ \cite[Thm. 6.2.(iii)]{TWW}.
\end{proof}

\end{document}